\documentclass[12pt]{amsart}
\usepackage{amssymb,amsmath,amsthm}
\usepackage{color}
\definecolor{forestgreen}{cmyk}{0.91,0,0.88,0.12}
\usepackage{enumitem}
	\setenumerate{leftmargin=0.42cm}
\usepackage{psfrag,graphicx}
\usepackage{pst-math}
\usepackage{pstricks}
\usepackage{pst-plot}
\usepackage{pst-slpe}
\usepackage{pst-all}
\def\eps{\varepsilon}
\def\A {\mathbb{A}}
\def\B {\mathfrak{B}}

\def\R {\mathbb{R}}

\def\H {\mathcal{H}}
\def\HH {H}
\def\CC {\mathbb{C}}
\def\D {{\rm dom}}
\def\ZZ {\mathfrak{Z}}

\def\Re {\mathfrak{Re\,}}

\def\e{{\rm e}}
\def\ii{{\rm i}}
\def\d{{\rm d}}
\def\ddt{\frac{\d}{\d t}}

\def \ds {\displaystyle}
\def \and {{\qquad\text{and}\qquad}}
\def \norm {\boldsymbol{|}}

\def \l {\langle}
\def \r {\rangle}
\def \wws {9cm}
\def \wwp {7cm}
\def \au {\rm}
\def \ti {\it}
\def \jou {\rm}
\def \bk {\it}
\def \no#1#2#3 {{\bf #1} (#3), #2.}
\def \eds#1#2#3 {#1, #2, #3.}

\newtheorem{proposition}{Proposition}[section]
\newtheorem{theorem}[proposition]{Theorem}
\newtheorem{corollary}[proposition]{Corollary}
\newtheorem{lemma}[proposition]{Lemma}
\theoremstyle{definition}

\newtheorem{remark}[proposition]{Remark}

\numberwithin{equation}{section}

\title[Second order equations with dissipation]
{Second order linear evolution\\ equations with general dissipation}

\author[F. Dell'Oro and V. Pata]
{Filippo Dell'Oro and Vittorino Pata}
\address{Politecnico di Milano - Dipartimento di Matematica
\newline\indent
Via Bonardi 9, 20133 Milano, Italy}
\email{filippo.delloro@polimi.it}
\email{vittorino.pata@polimi.it}
\begin{document}

\begin{abstract}
The contraction semigroup $S(t)=\e^{t\A}$ generated by the abstract linear dissipative evolution equation
$$
\ddot u + A u + f(A) \dot u=0
$$
is analyzed,
where $A$ is a strictly positive selfadjoint operator and $f$ is an arbitrary nonnegative
continuous function on the spectrum of $A$.
A full description of the spectrum of the infinitesimal generator $\A$ of $S(t)$ is provided.
Necessary and sufficient conditions for the stability, the semiuniform stability and the exponential stability of the semigroup
are found, depending on the behavior of $f$
and the spectral properties of its zero-set.
Applications to wave, beam and plate equations with fractional damping are also discussed.
\end{abstract}

\maketitle

\smallskip
\begin{center}
\begin{minipage}{12cm}
\small
\tableofcontents
\end{minipage}
\end{center}
\newpage

\section{Introduction}

\noindent
Let $(\HH,\langle \cdot , \cdot \rangle, \|\cdot\|)$ be a separable complex Hilbert space, and let
$$
A :\D(A) \subset \HH \to \HH
$$
be a strictly positive selfadjoint linear operator with inverse $A^{-1}$ not necessarily compact.
Let also
$$f:\sigma(A)\to[0,\infty)$$
be a nonnegative continuous function on the spectrum $\sigma(A)$ of $A$.
Since $A$ is strictly positive selfadjoint, $\sigma(A)$ is a nonempty closed subset of $\R^+=(0,\infty)$.
Moreover, $\sigma(A)$
is compact if and only if $A$ is a bounded operator on $\HH$.

For $t>0$, we consider the abstract second order evolution equation in the unknown variable $u=u(t)$
\begin{equation}
\label{MAIN}
\ddot u + A u + f(A) \dot u=0,
\end{equation}
where $u(0)$ and $\dot u(0)$ are understood to be assigned initial data and the {\it dot} stands for derivative
with respect to $t$. Here, $f(A)$ is the selfadjoint operator constructed via the functional calculus of $A$,
namely,
$$
f(A) = \int_{\sigma(A)} f(s) \,\d E_A (s)
$$
being $E_A$ the spectral measure of $A$ (see e.g.\ \cite{Rudin}).
More details on the functional calculus will be given in Section \ref{SEC3}.

\smallskip
Equation \eqref{MAIN} falls within a general class of models introduced in \cite{CR} to account for the dissipative
mechanism acting in elastic systems. The operator $A$ is usually called {\it elastic operator}
while $f(A)$, replaced in \cite{CR} by a more general nonnegative selfadjoint operator $B$,
is called {\it dissipation operator}. In the last decades, these models have been the object
of intensive mathematical investigations, and nowadays the current
literature on the subject is rather vast. When the dissipation operator is comparable with
the power $A^\vartheta$ for some $\vartheta\in[0,1]$
of the elastic operator $A$ (i.e.\ when the function $f(s)$ controls and is controlled by $s^\vartheta$),
then the associated solution semigroup is known to be exponentially stable and, in addition,
analytic for $\vartheta \in [\tfrac12,1]$ and
of Gevrey type for $\vartheta \in (0,\tfrac12)$;
see e.g.\ \cite{CR,CT1,CT2,CTFractional,HUA1,HUA,HUA2} and the more recent
contributions \cite{HO,JT0,JT,LTBOOK,LTMatrix,LIU1,LIU2,MU}, among many others.
At the same time, when $\vartheta \notin [0,1]$, the exponential stability is lost.
In particular, for $\vartheta<0$, the solution semigroup is known to be semiuniformly stable (a notion of stability weaker
than the exponential one), with optimal
polynomial decay rate of order $\tfrac{1}{2|\vartheta|}$ (see \cite{FAT,LiuZAMP}). The case $\vartheta>1$ has been analyzed
in the very recent paper \cite{HTAMS}, where well-posedness and
further regularity properties of the solutions have been discussed.
The above-mentioned results are highly nontrivial, and require
the exploitation of several abstract tools from the theory of linear semigroups, combined with quite delicate sharp computations.

On the other hand, when the dissipation operator is not comparable with $A^\vartheta$, namely, when
the function $f$ is allowed to exhibit an arbitrary (and not necessarily polynomial) behavior,
the picture becomes even more challenging,
and additional difficulties arise. In this situation,
the literature about the longterm properties of equation \eqref{MAIN} is poorer and mainly devoted to
the study of conditions under which all the solutions decay exponentially
to zero (see e.g.\ \cite[Chapter VI]{ENGNAG} and the further papers \cite{BATK,GOLD1,GOLD,GRIN1,GRIN2}). Roughly speaking, these contributions
tell that exponential stability occurs whenever the following two assumptions hold
(plus possibly some extra conditions varying from paper to paper):
\begin{itemize}
\item[{\rm (i)}] the dissipation operator is bounded below, namely,
$\inf_{s\in\sigma(A)} f(s)>0$; and
\smallskip
\item[{\rm (ii)}] the dissipation operator is subordinate to $A$, namely,
$\sup_{s\in \sigma(A)} f(s)/ s <\infty$.
\end{itemize}
Note that within (i)
the function $f$ does not vanish on $\sigma(A)$.

\smallskip
In light of the discussion above two natural questions arise:
\begin{itemize}
\smallskip
\item[$\diamond$] {\it What can be said on
the stability of \eqref{MAIN} when
the dissipation operator is not necessarily comparable with $A^\vartheta$ and not
necessarily bounded below,
nor subordinate to $A$?}
\smallskip
\item[$\diamond$] {\it In particular, what happens when the function $f$ vanishes in some points of $\sigma(A)$?}
\end{itemize}

\smallskip
The aim of the present work is to address these issues. After proving the existence of the contraction semigroup $S(t)$ of solutions
for a general nonnegative continuous function $f$ [see Theorem~\ref{EU}],
we show that
$S(t)$ is always stable, i.e.\ all single trajectories decay to zero, provided that
the zero-set of $f$
\begin{equation}
\label{ZERO}
\ZZ = \big\{s \in \sigma(A) : \, f(s)=0 \big\}
\end{equation}
has null spectral measure
and is at most countable [see Theorem \ref{stability}]. In fact, this condition is sharp:\
when $\ZZ$ has positive spectral measure, solutions with positive constant energy pop up.
These results are attained via an explicit description of the spectrum of the infinitesimal
generator of the semigroup [see Theorems \ref{zerospect} and \ref{spec}].
Such a description, which seems to be new in the literature, besides having an interest by itself allows to prove the stability of $S(t)$
without assuming the compactness of the inverse operator $A^{-1}$ (or similar compactness conditions).
On the contrary, compactness conditions are typically used to apply the classical
Sz.-Nagy-Foias theory \cite{BEN,NF} or Jacobs-Glicksberg-deLeeuw-type theorems \cite[Chapter~5]{BattyLibroBis}.
In addition, we show that conditions (i)-(ii) above are actually necessary and sufficient in order for $S(t)$ to be exponentially stable
[see Theorem~\ref{stabexpteo}].
In particular, we provide an elementary proof
of the exponential stability of $S(t)$ which does not rely in any way on the linear structure of equation, and hence
can be exported to study nonlinear versions of \eqref{MAIN}.
We also analyze an intermediate notion of stability,
the so-called semiuniform stability, proving that $S(t)$ is semiuniformly stable if and only if the set $\ZZ$
is empty and assumption (ii) is satisfied [see Theorem \ref{tsus}]. Then,
we find the optimal polynomial semiuniform decay rate, again without assuming the compactness of the inverse operator $A^{-1}$
[see Theorems \ref{decayrateteo} and \ref{decayrateopt}].
We finally apply the results to some concrete physical models of waves,
beams and plates with fractional damping.

\section{Two Examples}
\label{seccont}

\noindent
In this section, we dwell on two particular (but relevant) instances of equation \eqref{MAIN}.
To this end, given a bounded domain $\Omega\subset \R^n$
with smooth boundary $\partial \Omega$, we introduce\footnote{It is understood that, in the real case, the results of this paper
apply by considering the natural complexifications of the involved spaces and operators.} 
the strictly positive selfadjoint Laplace-Dirichlet operator
on $L^2(\Omega)$
$$
L=-\Delta \qquad \text{with}\qquad \D(L) = H^2(\Omega) \cap H^1_0(\Omega),
$$
being $H^2(\Omega)$ and $H^1_0(\Omega)$ the standard Sobolev spaces on $\Omega$. For $r=0,1,2,\ldots,$ we also define the
Hilbert space (the index $r$ will be omitted whenever zero)
$$
V^r=\D(L^\frac{r}{2}),\qquad \l u,v\r_{V^r} = \langle L^\frac{r}{2} u
, L^\frac{r}{2} v \rangle_{L^2(\Omega)},\qquad \| u\|_{V^r}=\| L^\frac{r}{2} u \|_{L^2(\Omega)}.
$$
In particular,
$$V^2=H^2(\Omega)\cap H_0^1(\Omega)
\,\,\subset\,\,
V^1=H_0^1(\Omega)
\,\,\subset\,\,
V=L^2(\Omega).$$

\subsection{Abstract wave equations with fractional damping}
We consider the abstract
wave equation
\begin{equation}
\label{polyabseq}
\ddot u + A u + A^\vartheta \dot u=0,
\end{equation}
that is, equation \eqref{MAIN} with a damping of the form
$$f(s)=s^\vartheta,\quad\, \vartheta \in \R.$$
In particular, the values
$\vartheta=0,1$ yield to the so-called weakly damped wave equation
$$
\ddot u +A u + \dot u=0,
$$
and the so-called strongly damped wave equation
$$
\ddot u +A u + A \dot u=0,
$$
respectively. A concrete realization of \eqref{polyabseq} is the boundary-value problem
in the unknown variable $u=u(\boldsymbol{x},t):\Omega\times \R^+ \to \R$
\begin{equation}
\label{concvardamp}
\begin{cases}
\partial_{tt}u- \Delta u + (-\Delta)^\vartheta \partial_t u=0,\\
u(\boldsymbol{x},t)_{|\boldsymbol{x}\in\partial\Omega}=0,
\end{cases}
\end{equation}
corresponding to the choice $H=V$ and $A=L$.

\subsection{Beams and plates}
For $\vartheta \in \R$ and $\omega\geq0$, we consider the evolution equation
in the unknown variable $u=u(\boldsymbol{x},t):\Omega\times \R^+ \to \R$
\begin{equation}
\label{platefraca}
\partial_{tt}u -\omega \Delta\partial_{tt}u+ \Delta^2 u+(-\Delta)^\vartheta \partial_t u=0,
\end{equation}
complemented with the hinged boundary conditions
\begin{equation}
\label{platefracaBC}
u(\boldsymbol{x},t)_{|\boldsymbol{x}\in\partial\Omega}=\Delta u(\boldsymbol{x},t)_{|\boldsymbol{x}\in\partial\Omega}=0,
\end{equation}
which rules the dynamics of a hinged beam (for $n=1$) or plate (for $n=2$) subject
to fractional dissipation. According to the formalism introduced above, \eqref{platefraca}-\eqref{platefracaBC}
read
\begin{equation}
\label{platefrac}
\partial_{tt}u +\omega L\partial_{tt}u+ L^2 u+L^\vartheta \partial_t u=0.
\end{equation}
In order to rewrite the latter equation in the abstract form~\eqref{MAIN}, we shall treat separately
the two cases $\omega=0$ and $\omega>0$.

\smallskip
\noindent
$\bullet$ If $\omega=0$, i.e.\ the rotational inertia is neglected, equation \eqref{platefrac} reduces to
\begin{equation}
\label{nonrotrwe}
\partial_{tt}u + L^2 u+L^\vartheta \partial_t u=0.
\end{equation}
The above
is nothing but the particular realization of~\eqref{MAIN} corresponding to the choice $H=V$, $A=L^2$ with $\D(A)=V^4$, and
$$f(s)=s^\frac{\vartheta}{2}.$$

\noindent
$\bullet$ If $\omega>0$, equation \eqref{platefrac} can be rewritten as
\begin{equation}
\label{rotrwe}
\partial_{tt}u + (1+\omega L)^{-1} L^2 u+(1+\omega L)^{-1}L^\vartheta \partial_t u=0.
\end{equation}
Endowing the space $V^1=H_0^1(\Omega)$ with the equivalent Hilbert norm
$$
\norm u\norm_{V^1}=\| (1+\omega  L)^\frac{1}{2} u \|_{L^2(\Omega)},
$$
we now choose $H=V^1$. It is then readily seen that the linear operator on $H$ (with the norm above)
$$
A = (1+\omega L)^{-1} L^2 \qquad  \text{with}\qquad \D(A) = V^3
$$
is strictly positive selfadjoint. Moreover, calling
\begin{equation}
\label{effe}
f(s)= s \bigg(\frac{\omega s + \sqrt{\omega^2 s^2 + 4 s}}{2} \bigg)^{\vartheta-2},
\end{equation}
by means of direct calculations we find the equality
$$
f(A) = (1+\omega L)^{-1}L^\vartheta.
$$
In conclusion, within these choices, equation \eqref{rotrwe} takes the form \eqref{MAIN}.

\section{The Spectral Measure of $A$}
\label{SEC3}

\noindent
Along the paper, the functional calculus of $A$ will be extensively used.
Recall that a {\it spectral measure} on a closed set $\Omega\subset\R$ is a map
$$
E: \mathfrak{B}(\Omega) \to P(H),
$$
defined on the Borel $\sigma$-algebra $\mathfrak{B}(\Omega)$ of $\Omega$ with values in the
space $P(H)$ of selfadjoint projections in $H$, satisfying the following
properties:
\begin{itemize}
\item $E(\emptyset)=0\,$ and $\,E(\Omega)=1$.
\smallskip
\item $E(\omega_1\cap\omega_2)=E(\omega_1)E(\omega_2)$, for all $\omega_1,\omega_2\in \mathfrak{B}(\Omega)$.
\smallskip
\item If $\omega_1\cap\omega_2=\emptyset\,$ then
$\,E(\omega_1\cup\omega_2)=E(\omega_1)+E(\omega_2)$.
\smallskip
\item For every $u,v\in H$ the set function $\mu_{u,v}$ on $\B(\Omega)$ defined by
$$\mu_{u,v}(\omega)=\langle E(\omega)u,v\rangle$$
is a complex measure.
\end{itemize}
By the Spectral Theorem (see e.g.\ \cite{Rudin}), there exists a unique spectral measure $E_A$
on the set $\Omega=\sigma(A)$, called the spectral measure of $A$, such that
$$
\l A u , v \r = \int_{\sigma(A)} s \,\d \mu_{u,v}
$$
for all $u \in \D(A)$ and $v \in H$. The integral representation above is usually written for short as
$$
A = \int_{\sigma(A)} s \,\d E_A (s).
$$
In addition,
for every continuous complex-valued
function $\phi$ on $\sigma(A)$, we can define the linear operator $\phi(A)$ by
$$
\phi(A) = \int_{\sigma(A)} \phi(s) \,\d E_A (s)
$$
with dense domain
$$
\D(\phi(A)) = \Big\{u\in H :\int_{\sigma(A)} |\phi(s)|^2 \,\d \mu_{u,u} (s)<\infty \Big\}.
$$
It is well-known that $\phi(A)$ is a densely defined closed operator. Besides, $\phi(A)$ is
selfadjoint if and only if $\phi$ is real-valued. Further properties of $\phi(A)$ read as follows:
\begin{itemize}
\item for every $u \in \D(\phi(A))$, we have the equality
$$
\|\phi(A)u \|^2 =  \int_{\sigma(A)} |\phi(s)|^2 \,\d \mu_{u,u} (s).
$$
\item $\phi(A)$ is bounded if and only if $\phi$ is bounded.
In which case,
$$
\|\phi(A)\|= \sup_{s \in\sigma(A)}|\phi(s)|.
$$
\item $\phi(A)$ is bounded below if and only if
$$\inf_{s \in\sigma(A)}|\phi(s)|>0.
$$
\end{itemize}
It is apparent to see that $\D(\phi(A))$ endowed with the graph norm
$$
\|u\|_{\D(\phi(A))}^2 = \|u\|^2 + \|\phi(A)u\|^2
$$
is a Hilbert space. In particular, when $\phi(A)$ is bounded below, there exists $\mathfrak{c}>1$ such that
$$
\|u\|_{\D(\phi(A))}^2 \leq \mathfrak{c} \|\phi(A)u\|^2.
$$
Hence, the seminorm $\|\phi(A)u \|$
is actually a norm, equivalent to the graph norm.

\section{Functional Setting and Notation}

\noindent
For $r\in\R$, we define the family of Hilbert spaces ($r$ is always omitted whenever zero)
$$
\HH^r=\D(A^\frac{r}{2}),\qquad \langle u,v\rangle_r = \langle A^\frac{r}{2} u
, A^\frac{r}{2} v \rangle,\qquad \| u\|_r=\| A^\frac{r}{2} u \|.
$$
If $r>0$, it is understood that
$\HH^{-r}$ denotes the completion of the domain, so that
$\HH^{-r}$ is the dual space of $\HH^{r}$. Accordingly, the symbol $\l \cdot , \cdot \r$
will also stand for duality product between $\HH^{r}$ and $\HH^{-r}$.
Setting
$$
s_0= \min \{s :s\in \sigma(A)  \}>0,
$$
for every $r_1<r_2$ we have
the Poincar\'e inequality (which follows at once from the functional calculus)
\begin{equation}
\label{POINC}
\|u\|_{r_1}^2 \leq s_0^{r_1-r_2}\|u\|_{r_2}^2,\quad \forall u \in \HH^{r_2}.
\end{equation}
In particular, the continuous and dense (but not necessarily compact) inclusion
$$
\HH^{r_2} \subset \HH^{r_1}
$$
holds true. Along the paper, the Poincar\'e inequality, as well as the Young and H\"older inequalities,
 will be tacitly used several times.
We conclude by defining the phase space of our problem
$$
\H = \HH^1 \times \HH
$$
endowed with the standard Hilbert product norm
$$
\|(u,v)\|_\H^2 = \|u\|_1^2 + \|v\|^2.
$$

\section{The Linear Operator ${\A}$}

\noindent
In view of rewriting equation \eqref{MAIN} as a first order ODE on $\H$,
we introduce
the linear operator $\A:\D(\A)\subset\H\to\H$ defined as
$$
\A
\left(\begin{matrix}
u\\
v
\end{matrix}\right)
= \left(
\begin{matrix}
v\\
- A u - f(A)v
\end{matrix}
\right)
$$
with domain
$$
\D(\A)=
\left\{ (u,v)\in\H \left|
\begin{array}{c}
v\in \HH^1\\
Au + f(A)v \in \HH
\end{array}\right.
\right\}.
$$

\begin{remark}
As customary, as we did in the definition above of the domain of $\A$,
whenever a vector $u\in\HH$ does not belong to $H^2$ we still write $A u$ to mean
the element of the dual space $H^{-2}$ acting as
$$
(Au) [w] \doteq \l A w, u  \r,\quad \forall w \in H^2.
$$
Analogously, whenever a vector $v\in\HH$ does not belong to $\D(f(A))$, we still write $f(A)v$ to mean
the element of the dual space $\D(f(A))^*$ acting as
$$
(f(A)v) [w] \doteq \l f(A) w, v\r,\quad \forall w \in \D(f(A)).
$$
\end{remark}

\begin{remark}
It is readily seen that $\D(\A)$ is a dense subset of $\H$,
as it contains the dense subspace of $\H$
$$
\HH^2 \times \D(q(A))\qquad\text{where}\qquad q(s)= \sqrt{s + [f(s)]^2}.
$$
Besides, it is apparent to verify that if (and only if)
$\sup_{s\in\sigma(A)} f(s)/\sqrt{s}<\infty$
then the domain of $\A$ factorizes as
$$
\D(\A) = H^2 \times H^1.
$$
\end{remark}

\begin{remark}
\label{chebelremark}
The operator $\A$ is closed. This can either be proved directly, or deduced as a consequence of the fact that
$\A$ is the infinitesimal generator of a contraction semigroup (as shown in the following Theorem \ref{EU}).
Moreover, it is apparent that $\A$ is injective.
\end{remark}

If a pair $(u,v)$ belongs to the domain of $\A$, then the variables inherit additional
regularity. In particular, we have the following result.

\begin{lemma}
\label{regdom}
Let $(u,v)\in\D(\A)$ be arbitrarily given. Then
$$v\in \D(\sqrt{f(A)}).$$
\end{lemma}

\begin{proof}
Exploiting the conditions $Au+f(A)v \in H$ and $u \in H^1$, we find at once the relation
$A^{-\frac{1}{2}}f(A)v \in \HH$, meaning that
$v \in \D(g(A))$ where  $g(s)=f(s)/\sqrt{s}$.
Since $v\in H^1$, by the definition of $\D(\A)$, an application of the H\"older inequality yields
\begin{align*}
\int_{\sigma(A)} f(s) \,\d \mu_{v,v} (s)
&\leq \left(\int_{\sigma(A)} s \,\d \mu_{v,v} (s) \right)^{\frac{1}{2}}
\left(\int_{\sigma(A)} \frac{[f(s)]^2}{s} \,\d \mu_{v,v} (s) \right)^{\frac{1}{2}}\\
\noalign{\vskip1mm}
&=\|v\|_1 \|g(A)v\|
<\infty.
\end{align*}
The estimate above tells that $v\in\D(\sqrt{f(A)})$, as claimed.
\end{proof}

\begin{remark}
Actually, from the proof of Lemma \ref{regdom} we infer that the variable
$v$ belongs to the (more regular) space $H^1 \cap \D(g(A))\subset \D(\sqrt{f(A)})$.
\end{remark}

We conclude the section by showing that the (densely defined) operator $\A$ is dissipative,
i.e.\ $\Re\langle \A z, z \rangle_{\H}\leq 0$ for all $z\in\D(\A)$.

\begin{theorem}
\label{teodiss}
The dissipativity relation
\begin{equation}
\label{dissip}
\Re\langle \A z, z \rangle_{\H}=-\|\sqrt{f(A)}v\|^2\leq 0
\end{equation}
holds for every $z=(u,v) \in \D(\A)$.
\end{theorem}

\begin{proof}
The thesis is readily obtained by direct calculations, and recalling Lemma \ref{regdom}.
\end{proof}


\section{The Spectrum of $\A$}

\noindent
In this section, we describe the spectrum $\sigma(\A)$ of the (closed) operator~$\A$.
Besides having some interest by itself, such a description will play a crucial role in the analysis of the asymptotic properties
of \eqref{MAIN}. We first state a necessary and sufficient condition for $0\not\in \sigma(\A)$.

\begin{theorem}
\label{zerospect}
The operator $\A$ is bijective, i.e.\ $0\notin \sigma(\A)$, if and only if
\begin{equation}
\label{supsup}
\sup_{s\in \sigma(A)} \frac{f(s)}{s} <\infty.
\end{equation}
\end{theorem}

\begin{proof}
The (injective) operator $\A$ is bijective if and only if for any given $\hat z = (\hat u, \hat v)\in \H$ the equation
\begin{equation}
\label{zerores}
\A z = \hat z
\end{equation}
admits a (unique) solution $z=(u,v)\in \D(\A)$. Componentwise, this translates
into
$$
\begin{cases}
v=\hat u,\\
A u + f(A)v = - \hat v.
\end{cases}
$$
Substituting the first equation into the second one, we get
$$
u = -A^{-1}\hat v - f(A)A^{-1} \hat u.
$$
Since $A^{-1}\hat v \in \HH^1$, we have that $u\in\HH^1$ for every given $\hat u\in \HH^1$ if and only if
$f(A)A^{-1}$ is a bounded operator. It amounts to saying that condition \eqref{supsup} holds true. In such a case, the couple
$$(-A^{-1}\hat v - f(A)A^{-1} \hat u,\hat u)\in\D(\A)$$ is the unique solution
to \eqref{zerores}.
\end{proof}

\begin{remark}
Observe that \eqref{supsup} is automatically satisfied when $A$ is a bounded operator,
for $\sigma(A)$ is compact and $f$ is continuous. Hence, in this situation, it is always true that $0$ belongs to the resolvent set
$\rho(\A)$ of $\A$.
\end{remark}

We now provide a characterization of $\sigma(\A)\setminus\{ 0\}$. To this end,
for every fixed $s\in\sigma(A)$, we introduce the pair of complex numbers
$$
\xi_s^{\pm}=
\begin{cases}
\ds-\frac{f(s)}{2} \pm \frac{\sqrt{[f(s)]^2-4s}}{2} &\quad\text{if } f(s)\geq 2\sqrt{s}\,,\\
\noalign{\vskip1mm}
\ds -\frac{f(s)}{2} \pm \ii\frac{\sqrt{4s-[f(s)]^2}}{2} &\quad\text{if } f(s)< 2\sqrt{s}\,,
\end{cases}
$$
which are nothing but the solutions of the second order equation
$$
\xi^2+ f(s)\xi +s=0.
$$
We also consider the (possibly empty) subset of $\R^+$
\begin{equation}
\label{lambdaset}
\Lambda = \Big\{\ell>0:\, \exists \, s_n\in\sigma(A)
\,\,\,\text{such that}\,\,\, s_n\to\infty \,\,\,\text{and}\,\,\, \lim_{n\to\infty} \frac{f(s_n)}{s_n}=\ell\Big\}.
\end{equation}
The result reads as follows.

\begin{theorem}
\label{spec}
We have the equality
$$
\sigma(\A)\setminus\{ 0\} = \bigcup_{s\in\sigma(A)} \big\{\xi_s^{\pm} \big\}
\cup \bigcup_{\ell\, \in \Lambda} \Big\{-\frac{1}{\ell}\Big\}.
$$
\end{theorem}

\begin{proof}
Let $\xi \in \CC \setminus \{0\}$ and $\hat z = (\hat u, \hat v)\in\H$ be arbitrarily given. We look for a unique solution
$z= (u,v)\in\D(\A)$
to the resolvent equation
\begin{equation}
\label{res}
\xi z - \A z = \hat z.
\end{equation}
Written in components, we obtain the system
$$
\begin{cases}
\xi u - v=\hat u,\\
\xi v + A u + f(A)v =  \hat v.
\end{cases}
$$
Substituting the first equation into the second one, we find the expression
$$
\xi^2 v + A v + \xi f(A) v = A \hat w,
$$
having set
$$
\hat w = \xi A^{-1} \hat v - \hat u \in \HH^1.
$$
An exploitation of the functional
calculus now yields
$$
v = \int_{\sigma(A)} \frac{s}{\xi^2 + f(s)\xi+ s}\, \d E_A(s) \hat w.
$$
Thus $v\in\HH^1$ for any given $\hat w \in \HH^1$ if and only if
\begin{equation}
\label{failure}
\sup_{s\in\sigma(A)}\bigg| \frac{s}{\xi^2 + f(s)\xi+ s}\bigg| <\infty.
\end{equation}
This occurs if and only if
$$
\xi \notin  \bigcup_{s\in\sigma(A)} \big\{\xi_s^{\pm} \big\}\and
-\frac{1}{\xi} \notin \Lambda.$$
Indeed,
\eqref{failure} fails to hold if and only if there is a sequence $s_n\in\sigma(A)$ for which
\begin{equation}
\label{failure1}
\frac{\xi^2 + f(s_n)\xi+ s_n}{s_n}\to 0.
\end{equation}
If $s_n\not\to\infty$, then (up to a subsequence) $s_n$ converges to an element $s\in\sigma(A)$,
as the spectrum is a closed set. Hence  \eqref{failure1} becomes
simply
$$\xi^2 + f(s)\xi+ s=0,$$
meaning that $\xi=\xi_s^{+}$ or $\xi=\xi_s^{-}$.
If $s_n\to\infty$, then \eqref{failure1} implies that
$$\frac{f(s_n)}{s_n}\to-\frac1\xi,$$
that is, $-\frac1\xi\in\Lambda$.
Once we find $v$, we readily get
$$
u=\frac{v+\hat u}{\xi}\in \HH^1 \and Au + f(A)v = \hat v - \xi v \in \HH,
$$
meaning that
$(u,v)\in\D(\A)$
is the unique solution to \eqref{res}.
The theorem is proved.
\end{proof}

The next corollary will be needed in the sequel.

\begin{corollary}
\label{speccoro}
We have the equality
$$
(\sigma(\A)\setminus\{ 0\})\cap \ii \R= \bigcup_{s\in \ZZ} \pm \ii \sqrt{s},
$$
where $\ZZ$ is the zero-set of $f$ defined in \eqref{ZERO}.
\end{corollary}

\begin{proof}
Exploiting Theorem \ref{spec}, we learn at once that
$$
(\sigma(\A)\setminus\{ 0\})\cap \ii \R= \bigcup_{s\in\sigma(A)} \big\{\xi_s^{\pm} \big\} \cap \ii \R.
$$
Since $\xi_s^{\pm}\in \ii \R$ if and only if $s \in \ZZ$, and in such a case
$\xi_s^{\pm} = \pm \ii \sqrt{s}$, we are finished.
\end{proof}

\begin{remark}
\label{remeig}
By means of straightforward computations it is immediate to check that, if $s\in\sigma(A)$ is an eigenvalue of $A$, then $\xi_s^{\pm}$
are eigenvalues of $\A$.
\end{remark}

\section{The Contraction Semigroup}

\noindent
Introducing the state vector $z(t)=(u(t),v(t))$, we rewrite equation \eqref{MAIN} as the ODE in the phase space $\H$
\begin{equation}
\label{MAINabs}
\dot{z}(t) = \A z(t).
\end{equation}
The following holds.

\begin{theorem}
\label{EU}
The operator $\A$ is the infinitesimal generator of a contraction semigroup
$$S(t)=\e^{t\A}:\H \to \H.$$
\end{theorem}

As a consequence, for every given initial
datum $z_0 = (u_0,v_0)\in\H$ there exists a unique {\it mild solution} $z$
in the sense of Pazy \cite{Pazy} to equation \eqref{MAINabs}, explicitly given by the formula
$$
z(t)= S(t)z_0.
$$
The associated energy reads
$$
\mathcal{E}(t) = \frac12\|S(t)z_0 \|_\H^2.
$$
Moreover, if $z_0\in\D(\A)$, then $z(t)\in\D(\A)$ for all $t\geq0$, and the mild solution is actually a classical one.

\begin{proof}[Proof of Theorem \ref{EU}]
In light of the Lumer-Phillips theorem (see e.g.\ \cite{Pazy}),
the (densely defined) operator $\A$ generates a contraction semigroup on $\H$ if and only if
it is dissipative and $1-\A$ is onto. The first fact is assured by Theorem \ref{teodiss}.
In order to show the second instance, for an arbitrarily
given $\hat z = (\hat u, \hat v)\in \H$ we look for a solution $z=(u,v)\in \D(\A)$ to the equation
$$
z-\A z = \hat z.
$$
Written in components, the latter reads
$$
\begin{cases}
u - v=\hat u,\\
v + A u + f(A)v = \hat v.
\end{cases}
$$
Plugging the first equality into the second one, we find
$$
v + A v + f(A)v = A \hat w,
$$
where $\hat w = A^{-1} \hat v -\hat u \in \HH^1$.
Then, owing to the functional calculus, we get
$$
v=\int_{\sigma(A)} \frac{s}{1+s+f(s)} \,\d E_A (s) \hat w.
$$
Being the function $f$ nonnegative, we have
$$\sup_{s\in\sigma(A)} \frac{s}{1+s+f(s)}\leq 1,$$
meaning that $v\in \HH^1$. In turn, $u=v+\hat u\in\HH^1$ as well. Finally, by comparison,
$$
A u + f(A)v = \hat v - v \in \HH.
$$
This completes the argument.
\end{proof}

\section{The Conservative Case}
\label{interlude}

\noindent
For the sake of completeness, we preliminarily dwell on the conservative case, which is very well known in the literature.
Indeed, when the function $f$ vanishes
on the spectrum of $A$, the same as saying that $\ZZ=\sigma(A)$, equation \eqref{MAIN} reduces to
\begin{equation}
\label{casecons}
\ddot u + A u=0.
\end{equation}
Despite its simple form, it serves as a model for several physical phenomena.
With reference to the notation of Section \ref{seccont}, the simplest
example is the classical wave equation with Dirichlet boundary conditions
in the unknown $u=u(\boldsymbol{x},t):\Omega\times \R^+ \to \R$
$$
\begin{cases}
\partial_{tt}u- \Delta u =0,\\
u(\boldsymbol{x},t)_{|\boldsymbol{x}\in\partial\Omega}=0,
\end{cases}
$$
corresponding to the choice $H=V$ and $A=L$ in \eqref{casecons}.
Another model matching the abstract form \eqref{casecons} is the linear Klein-Gordon equation
arising in Relativity Theory
$$
\partial_{tt}u- \Delta u +m^2 u =0,
$$
in the unknown $u=u(\boldsymbol{x},t):\R^3\times \R^+ \to \R$, where $m>0$ (see e.g.\ \cite{GREN}).
It is readily seen that, choosing $H=L^2(\R^3)$ and
$$
A = -\Delta +  m^2 \qquad\text{with}\qquad
\D(A)=H^2(\R^3),
$$
the Klein-Gordon equation takes the form \eqref{casecons}. In this situation, the strictly positive selfadjoint operator $A$
does not have compact inverse.

The next result follows immediately from Theorems \ref{zerospect} and \ref{spec}, observing
that the set $\Lambda$ defined in \eqref{lambdaset} is always empty whenever $f\equiv0$.

\begin{theorem}
Assume that $f(s)=0$ for all $s\in\sigma(A)$. Then, the operator $\A$ is always bijective, and
its spectrum fulfils the equality
$$
\sigma(\A) = \bigcup_{s\in \sigma(A)} \pm \ii \sqrt{s}.
$$
In particular, $\sigma(\A)$ is entirely contained in the imaginary axis $\ii\R$.
\end{theorem}

\section{Stability}
\label{stabilitySection}

\noindent
In this section, we analyze the stability of $S(t)$. Recall that
$S(t)$ is said to be {\it stable} if, for every fixed $z_0\in \H$,
$$
\lim_{t\to\infty}\|S(t)z_0\|_{\H}=0.
$$
It turns out that this property depends dramatically on the structure of the set $\ZZ$ defined in \eqref{ZERO}.

\begin{theorem}
\label{stability}
The following hold:
\begin{itemize}
\item[{\rm (i)}] If $E_A(\ZZ)\neq 0$ (i.e.\ it is not the null projection), then there exist solutions with constant
positive energy.
In particular, the semigroup $S(t)$ is not stable.
\smallskip
\item[{\rm (ii)}] If $E_A(\ZZ)=0$ and the set $\ZZ$ is at most countable, then the semigroup $S(t)$ is stable.
In particular, this is always the case if $\ZZ=\emptyset$.
\end{itemize}
\end{theorem}

In order to show the theorem, we make use of the
famous {\it Arendt-Batty-Lyubich-V\~u stability criterion}~\cite{AB,LV}. Recall that,
denoting by $\sigma_{{\rm p}}(\A)$ the point spectrum of the infinitesimal generator $\A$,
the criterion reads as follows.

\begin{theorem}[Arendt-Batty-Lyubich-V\~u]
\label{ABC}
Assume that $\sigma_{{\rm p}}(\A)\cap \ii\R =\emptyset$ and $\sigma(\A) \cap \ii\R$ is at most countable. Then $S(t)$ stable.
\end{theorem}

We now proceed with the proof of Theorem \ref{stability}.

\begin{proof}[Proof of Theorem \ref{stability}]
Assume first that  $E_A(\ZZ)$ is a nonnull projection. In this situation,
we can select a unit vector
$w\in E_A(\ZZ)\HH$.
In particular, the probability measure $\mu_{w,w}$ is supported on $\ZZ$. Accordingly,
$$
\int_{\sigma(A)} [f(s)]^2 \d \mu_{w,w}(s)=0,
$$
meaning that $w\in\D(f(A))$ and $f(A)w=0$.
Then, by direct calculations, the solution to~\eqref{MAINabs}
with initial datum
$$
z_0=(0, A^{-\frac12}w) \in \D(\A)
$$
is given by
$$
z(t)=(\sin (t\sqrt{A}) A^{-1} w, \cos (t\sqrt{A}) A^{-\frac12} w).
$$
Such a solution $z$ has constant energy, for
\begin{align*}
\mathcal{E}(t) &=
\frac12\int_{\sigma(A)} \frac{1}{s}\, \big[|\sin(t \sqrt{s})|^2 + |\cos(t \sqrt{s})|^2\big]  \d \mu_{w,w}(s)
=\frac12\|A^{-\frac12}w\|^2=\mathcal{E}(0).
\end{align*}
The proof of item (i) is finished.
In order to show (ii), we first prove that
$$E_A(\ZZ)=0 \quad\,\, \Rightarrow \quad\,\, \sigma_{{\rm p}}(\A)\cap \ii\R =\emptyset.$$
To this end, assume by contradiction that $\ii \lambda \in \sigma_{{\rm p}}(\A)$ for some $\lambda\in \R$.
Since $\A$ is injective (see Remark \ref{chebelremark}), we have $\lambda\neq0$.
Then, there is a nonnull vector
$z=(u,v)\in \D(\A)$ satisfying
$$
\ii \lambda z - \A z=0.
$$
Componentwise, the equality above reads
\begin{equation}
\label{sss1}
\begin{cases}
\ii \lambda u - v=0,\\
\ii \lambda v + A u + f(A) v=0.
\end{cases}
\end{equation}
Invoking~\eqref{dissip}, we obtain
$$
0 = \Re \l \ii \lambda z - \A z , z\r_\H= \|\sqrt{f(A)} v\|^2=
\int_{\sigma(A)} f(s)\, \d \mu_{v,v}(s),
$$
implying that
$$v \in E_A(\ZZ)H=0.$$
Making use of the first equation of \eqref{sss1} and the fact that $\lambda\neq0$, we find  $u=0$, reaching the desired contradiction $z=0$.
At this point, Corollary \ref{speccoro} together with the assumption
that $\ZZ$ is either finite or countable ensure that $\sigma(\A) \cap \ii\R$ is either
finite or countable. The abstract Theorem \ref{ABC} then allows to conclude.
\end{proof}

Theorem \ref{stability} tells in particular that $S(t)$ is stable whenever $\ZZ=\emptyset$, but in general
the converse is not true. Nevertheless, when the operator $A^{-1}$ is compact,
the sufficient condition $\ZZ=\emptyset$ turns out to be necessary as well.

\begin{corollary}
Assume that $A^{-1}$ is a compact operator. If $S(t)$ is stable then $\ZZ=\emptyset$.
\end{corollary}

\begin{proof}
When $A^{-1}$ is compact, it is well known that the spectrum $\sigma(A)$ is made of a sequence of eigenvalues tending to infinity.
In particular,
$$
E_A(\{ s\}) \neq 0,\quad \forall s \in \sigma(A).
$$
This clearly yields $E_A (\ZZ) \not=0$ whenever $\ZZ\not=\emptyset$. Due to Theorem \ref{stability}, the semigroup $S(t)$ is not stable.
\end{proof}

\begin{remark}
\label{remopen}
In the case where $A^{-1}$ is not compact,
the question whether or not $S(t)$ is stable if $E_A(\ZZ)=0$ and the set
$\ZZ$ is uncountable remains open.
\end{remark}

\section{Exponential Stability}
\label{ExpSec}

\noindent
A much stronger notion of stability is the exponential (or uniform) one.
Recall that $S(t)$ is said to be {\it exponentially stable} if there exist
$\varkappa>0$ and $M\geq 1$ such that
$$
\|S(t)z_0\|_{\H}\leq M\e^{-\varkappa t}\|z_0\|_\H,
$$
for all $z_0\in\H$. Exponential stability is equivalent to the fact that
the operator norm $\|S(t)\|_{L(\H)}$ goes to zero as $t\to\infty$. In turn, this is the same as saying that
$\omega_*<0$, where $\omega_*$ is the
{\it growth bound} of $S(t)$, defined as
$$\omega_*=\inf\big\{\omega\in\R:\,\|S(t)\|_{L(\H)}\leq M\e^{\omega t}\big\}
$$
for some $M=M(\omega)$. From the Hille-Yosida Theorem and the boundedness of $S(t)$, it follows at once the relation
\begin{equation}
\label{sigom}
\sigma_*\leq \omega_*\leq0,
\end{equation}
where
$$\sigma_*=\sup\big\{\Re\xi:\,\xi\in\sigma(\A)\big\}
$$
is the {\it spectral bound} of $\A$ (see e.g.\ \cite{Pazy} for more details).

\begin{theorem}
\label{stabexpteo}
The semigroup $S(t)$ is exponentially stable if and only if
\begin{equation}
\label{condexpstab}
\inf_{s\in\sigma(A)} f(s)>0 \and \sup_{s\in \sigma(A)} \frac{f(s)}{s} <\infty.
\end{equation}
\end{theorem}

We shall prove separately the necessity and the sufficiency parts.

\begin{proof}[Proof of Theorem \ref{stabexpteo} (Necessity)]
The strategy consists in showing that, if \eqref{condexpstab} is not satisfied,
then $\sigma_*\geq 0$. Due to \eqref{sigom}, the latter condition implies that $\omega_*=0$, i.e.\ $S(t)$ is not exponentially stable.
Indeed, when
$$\sup_{s\in \sigma(A)} \frac{f(s)}{s} =\infty,$$
we learn from Theorem \ref{zerospect} that $0\in\sigma(\A)$, namely, $\sigma_*\geq0$.
On the other hand, when
$$
\inf_{s\in\sigma(A)} f(s)=0,
$$
there exists a sequence $s_n\in \sigma(A)$ such that
$$f(s_n)\to0$$
as $n\to\infty$. Since the spectrum of $A$ is (positive and) away from zero,
it is clear that for all $n$ large we have $f(s_n)< 2\sqrt{s_n}$.
Exploiting Theorem \ref{spec},
the complex numbers
$$
\xi_{n}^{\pm}=
-\frac{f(s_n)}{2} \pm \ii\frac{\sqrt{4s_n-[f(s_n)]^2}}{2}
$$
belong to $\sigma(\A)$. Since
$\Re\xi_{n}^{\pm}\to 0$,
we conclude again that $\sigma_*\geq 0$.
\end{proof}

\begin{proof}[Proof of Theorem \ref{stabexpteo} (Sufficiency)]
Let $z_0\in\D(\A)$ be an arbitrarily fixed initial datum, and let
$$
z(t)=(u(t),\dot{u}(t))= S(t)z_0 \in \D(\A)
$$
be the corresponding solution. Since $\D(\A)$ is a dense subset of $\H$, in order to reach the
conclusion it is enough showing that the associated energy $\mathcal{E}(t)$ fulfills
$$
\mathcal{E}(t) \leq M^2 \mathcal{E}(0) \e^{-2\varkappa t}
$$
for some $\varkappa>0$ and $M\geq 1$ independent of $z_0$.
Along the proof, $c>0$ will denote a {\it generic} positive
constant depending only on the structural quantities of the problem and independent of $z_0$.
Multiplying equality \eqref{MAINabs} by $z$ in $\H$, taking the real part and exploiting Theorem~\ref{teodiss} we
find the identity
$$
\ddt \mathcal{E}=\Re\langle \A z, z \rangle_{\H}=-\|\sqrt{f(A)}\dot u\|^2.
$$
Invoking the first condition of \eqref{condexpstab}, it is apparent to see that
$$
\|\sqrt{f(A)}\dot u\|^2 = \int_{\sigma(A)} f(s) \,\d \mu_{\dot u,\dot u} (s) \geq c \|\dot u \|^2.
$$
Hence, we get the differential inequality
\begin{equation}
\label{primaexp}
\ddt \mathcal{E} + c \|\dot u \|^2 \leq0.
\end{equation}
Next, we introduce the auxiliary functional
$$\Phi(t)=  \frac12 \|\sqrt{f(A)}u(t)\|^2+\Re \l \dot u(t) , u(t) \r.$$
Due the Poincar\'e inequality \eqref{POINC} and the second condition of \eqref{condexpstab},
\begin{equation}
\label{equiv}
|\Phi| \leq  \frac12 \int_{\sigma(A)} f(s) \,\d \mu_{u, u} (s)+c\mathcal{E} \leq
c \int_{\sigma(A)} s \,\d \mu_{u, u} (s)+c\mathcal{E}
\leq c \mathcal{E}.
\end{equation}
Moreover, by means of direct calculations,
\begin{equation}
\label{secondaexp}
\ddt \Phi + \|u\|_1^2 = \|\dot u \|^2.
\end{equation}
At this point, for all $\eps>0$, we define the energy-like functional
$$
\Lambda_\eps(t)= \mathcal{E}(t) + \eps \Phi(t).
$$
An exploitation of \eqref{equiv} yields
$$
\frac12 \mathcal{E}(t) \leq \Lambda_\eps(t) \leq 2\mathcal{E}(t)
$$
for every $\eps>0$ small enough, meaning that $\Lambda_\eps$ is equivalent to $\mathcal{E}$. Thus,
collecting \eqref{primaexp} and \eqref{secondaexp}, and fixing the parameter $\eps>0$
sufficiently small, we arrive at the differential inequality
$$
\ddt \Lambda_\eps + 2\varkappa \Lambda_\eps \leq 0
$$
for some $\varkappa>0$.
Applying the Gronwall lemma, and using once more the equivalence of the functionals $\Lambda_\eps$ and $\mathcal{E}$,
the conclusion follows.
\end{proof}


\section{Semiuniform Stability}

\noindent
Finally, we consider an intermediate notion of stability,
known as semiuniform stability. By definition, $S(t)$
is {\it semiuniformly stable} if there exists a nonnegative function $\psi(t)$ vanishing at infinity such that
\begin{equation}
\label{defdecay}
\|S(t) z_0 \|_\H \leq \psi(t) \| \A z_0\|_\H,\quad \forall z_0\in \D(\A).
\end{equation}
Since $S(t)$ is a bounded semigroup (actually, a contraction), it easily follows by density that
if $S(t)$ is semiuniformly stable then it is stable as well.
Instead, if $S(t)$ is exponentially stable then, as a consequence of \eqref{sigom},
its infinitesimal generator $\A$ is invertible with bounded inverse, as $0$ belongs to the resolvent set $\rho(\A)$ of $\A$.
This clearly implies that
$S(t)$ is semiuniformly stable with exponential rate, i.e.\
$$\psi(t)=M'\e^{-\varkappa t}\quad\text{with }\quad \varkappa,M'>0.$$
Precisely, with reference to the previous Section~\ref{ExpSec}, $M'=M\|\A^{-1}\|_{L(\H)}$
(recall that $\|\cdot\|_{L(\H)}$ denotes the operator norm).

\smallskip
The following criterion is due to Batty \cite{BattyBis,Batty1}.

\begin{theorem}[Batty]
\label{battscrit}
The (bounded) semigroup $S(t)$ is semiuniformly stable if and only if $\sigma(\A)\cap \ii\R=\emptyset$.
\end{theorem}

\begin{remark}
\label{bestfn}
In particular, when $S(t)$ is semiuniformly stable then $0\notin \sigma(\A)$.
Accordingly, for every ${z_0}\in \H$ we can write
$$
\|S(t)\A^{-1}{z_0} \|_\H \leq \psi(t)\|\A\A^{-1}{z_0}\|_{\H}=\|{z_0}\|_{\H}.
$$
This clearly implies that any function $\psi$ satisfying \eqref{defdecay} is subject to the constraint
$$
\psi(t)\geq\|S(t)\A^{-1}\|_{L(\H)}.
$$
At the same time, for all $z_0 \in \D(\A)$
$$
\|S(t)z_0 \|_\H \leq \|S(t)\A^{-1}\|_{L(\H)}\|\A z_0\|_{\H},
$$
meaning that among the $\psi$ complying with \eqref{defdecay}, the choice
$\psi(t)=\|S(t)\A^{-1}\|_{L(\H)}$ is the best possible.
\end{remark}

We now state the necessary and sufficient condition for the semiuniform stability of~$S(t)$.

\begin{theorem}
\label{tsus}
The semigroup $S(t)$ is semiuniformly stable if and only if
\begin{equation}
\label{cnssemi}
\ZZ=\emptyset \and \sup_{s\in \sigma(A)} \frac{f(s)}{s} <\infty.
\end{equation}
\end{theorem}

\begin{proof}
Collecting Theorem \ref{zerospect} and Corollary \ref{speccoro} we learn at once that $\sigma(\A)\cap \ii\R=\emptyset$
if and only if \eqref{cnssemi} holds true. Invoking Theorem \ref{battscrit}, we reach the thesis.
\end{proof}

\begin{remark}
\label{sesexp}
When the operator $A$ is bounded, condition \eqref{cnssemi} is equivalent to \eqref{condexpstab},
being the spectrum $\sigma(A)$ a compact set. Hence, in this situation, $S(t)$ is semiuniformly stable if and only
if it is exponentially stable.
\end{remark}

\section{Polynomial Decay Rates}
\label{POLYstabilitySection}

\noindent
When $S(t)$ is semiuniformly stable, it is
of great interest to describe the decay at infinity of $\psi(t)$ in \eqref{defdecay}.
In light of Remark~\ref{bestfn}, we concentrate on the particular function
$$\psi(t)=\|S(t)\A^{-1}\|_{L(\H)}.$$
It is an easy exercise to show that $\psi(t)$ is (Lipschitz) continuous.

\begin{theorem}
\label{decayrateteo}
Let $S(t)$ be semiuniformly stable. If there exists $\alpha>0$ such that
\begin{equation}
\label{furcont}
\inf_{s\in\sigma(A)} s^\alpha f(s)>0,
\end{equation}
then $\psi(t)$ decays at least polynomially
as $t^{-\frac{1}{2\alpha}}$, namely,
$$
\limsup_{t\to\infty} t^{\frac{1}{2\alpha}}\psi(t) <\infty.
$$
\end{theorem}

\begin{theorem}
\label{decayrateopt}
Let $S(t)$ be semiuniformly stable. If $A$ is unbounded and there exists $\beta>0$ such that
\begin{equation}
\label{furcont2}
\sup_{s\in\sigma(A)} s^\beta f(s) < \infty,
\end{equation}
then $\psi(t)$ decays at most polynomially
as $t^{-\frac{1}{2\beta}}$, namely,
$$
\limsup_{t\to\infty} t^{\frac{1}{2\beta}}\psi(t)>0.
$$
\end{theorem}

\begin{remark}
Actually, when $A$ is a bounded operator, the first result above is of little use, since
in that case (cf.\ Remark~\ref{sesexp}) $S(t)$ is semiuniformly stable if and only
if it is exponentially stable.
\end{remark}

\begin{remark}
Note that when $A$ is unbounded, then \eqref{furcont2} is incompatible with
the first condition in \eqref{condexpstab}. Besides, it is apparent that
$\alpha\geq \beta$ when both \eqref{furcont} and
\eqref{furcont2} are satisfied (again, under the assumption $A$ unbounded).
\end{remark}

Theorems \ref{decayrateteo} and \ref{decayrateopt} produce an immediate corollary.

\begin{corollary}
Let $S(t)$ be semiuniformly stable. If $A$ is unbounded and \eqref{furcont}-\eqref{furcont2}
simultaneously hold with $\alpha=\beta$,
then $\psi(t)$ decays polynomially
as $t^{-\frac{1}{2\alpha}}$,
and such a decay rate is optimal.
\end{corollary}

The proofs of Theorems \ref{decayrateteo} and \ref{decayrateopt} will be given in the next two sections.

\section{Proof of Theorem \ref{decayrateteo}}
\label{primero}

\noindent
The strategy consists in finding a polynomial estimate from above on the growth rate of the resolvent operator of $\A$
on the imaginary axis $\ii\R$, as the latter can be linked with the decay rate of $\|S(t)\A^{-1}\|_{L(\H)}$ making use of
the following abstract result due to Borichev and Tomilov~\cite{BT}.

\begin{theorem}[Borichev-Tomilov]
\label{BTpol}
Assume that $\sigma(\A)\cap \ii\R=\emptyset$, and let $\nu >0$ be fixed.
Then, we have
$$
\|(\ii\lambda - \A)^{-1}\|_{L(\H)} = {\rm O}(|\lambda|^\nu)\quad\, \text{as }\, |\lambda|\to \infty
$$
if and only if
$$
\|S(t)\A^{-1}\|_{L(\H)} = {\rm O}(t^{-1/\nu})\quad\, \text{as } t\to \infty.
$$
\end{theorem}

The needed estimate is contained in the next lemma.

\begin{lemma}
\label{aboveres}
Within the assumptions of Theorem \ref{decayrateteo}, we have ($\sigma(\A)\cap \ii\R=\emptyset$ and)
$$
\|(\ii \lambda - \A)^{-1}\|_{L(\H)}= {\rm O} (|\lambda|^{2\alpha})\quad\, \text{as } |\lambda|\to\infty.
$$
\end{lemma}

\begin{proof}
In what follows, $c\geq0$ will denote a {\it generic} constant depending only on the structural quantities of the problem.
Since $S(t)$ is semiuniformly stable, we know from Theorem~\ref{battscrit} that
$\sigma(\A)\cap \ii\R=\emptyset$. Besides,
due to Theorem \ref{tsus},
condition \eqref{cnssemi} holds true.

For every fixed $\lambda \in \R$
and $\hat z = (\hat u, \hat v) \in \H$, the resolvent equation
\begin{equation}
\label{rsus}
\ii\lambda z - \A z= \hat z
\end{equation}
has a unique solution
$$z=(\ii\lambda-\A)^{-1} \hat z.$$
In order to prove the lemma, it is enough showing the estimate
\begin{equation}
\label{goalsus}
\|z\|_\H \leq c |\lambda|^{2\alpha}\|\hat z\|_\H
\end{equation}
for every $|\lambda|\geq1$. To this end, writing \eqref{rsus} componentwise, we get the system
\begin{align}
\label{sus01}
&\ii\lambda u - v=\hat u,\\
\label{sus02}
&\ii\lambda v + A u + f(A) v =  \hat v.
\end{align}
In light of the dissipativity property
\eqref{dissip}, a multiplication in $\H$ of \eqref{rsus} with $z$ entails the control
\begin{equation}
\label{sus1}
\|\sqrt{f(A)}v\|^2 = \Re\langle (\ii\lambda-\A) z, z \rangle_{\H}
= \Re\langle \hat z, z \rangle_{\H} \leq \|z\|_\H \|\hat z \|_\H.
\end{equation}
Next, multiplying in $\HH$ equation \eqref{sus02} by $u$ and exploiting \eqref{sus01}, we obtain
\begin{align*}
\|u\|_1^2 &= \l \hat v , u \r - \ii \lambda \l v , u\r - \l \sqrt{f(A)} v, \sqrt{f(A)} u \r\\
& = \l \hat v , u \r + \l v , \hat u\r + \|v\|^2- \l \sqrt{f(A)} v, \sqrt{f(A)} u \r\\
& \leq \|v\|^2 + \| \sqrt{f(A)} u \|\| \sqrt{f(A)} v\| + c\|z\|_\H \|\hat z \|_\H.
\end{align*}
Invoking the second condition in \eqref{cnssemi},
$$
\| \sqrt{f(A)} u\|
\leq c \left(\int_{\sigma(A)} s \,\d \mu_{u,u} (s)\right)^{\frac12} = c\|u\|_1 \leq c \|z\|_\H.
$$
Hence, appealing to \eqref{sus1}, we find
\begin{equation}
\label{sus12}
\|u\|_1^2 \leq \|v\|^2+c \|z\|_\H^{3/2}\|\hat z\|_\H^{1/2} + c\|z\|_\H \|\hat z \|_\H.
\end{equation}
Moreover, due to \eqref{sus01}, it is also true that
\begin{equation}
\label{sus11}
\|v\|_1 \leq |\lambda| \|z\|_\H + \|\hat z \|_\H.
\end{equation}
At this point, we shall treat separately two cases.

\smallskip
\noindent
{\sc Case ${\alpha \leq 1}$.}
An application of \eqref{furcont} and \eqref{sus1}, together with the  H\"older inequality, yields
\begin{align*}
\|v\|^2 &\leq
\left(\int_{\sigma(A)} \frac{1}{f(s)} \,\d \mu_{v,v} (s) \right)^{\frac{1}{2}}
\left(\int_{\sigma(A)} f(s) \,\d \mu_{v,v} (s) \right)^{\frac{1}{2}}\\
&\leq c \left(\int_{\sigma(A)} s^\alpha \,\d \mu_{v,v} (s) \right)^{\frac{1}{2}}\|\sqrt{f(A)}v\|\\
&\leq c \left(\int_{\sigma(A)} s \,\d \mu_{v,v} (s) \right)^{\frac{\alpha}{2}}\|v\|^{1-\alpha}\|z\|_\H^{1/2}\|\hat z\|_\H^{1/2}\\\noalign{\vskip1.5mm}
&\leq c \|v\|_1^{\alpha}\|z\|_\H^{(3 - 2\alpha)/2}\|\hat z\|_\H^{1/2}.
\end{align*}
Estimating the term $\|v\|_1^{\alpha}$ with the aid of \eqref{sus11}, we reach the control
$$
\|v\|^2 \leq c |\lambda|^\alpha \|z\|_\H^{3/2}\|\hat z\|_\H^{1/2} + c \|z\|_\H^{(3-2\alpha)/2}\|\hat z\|_\H^{(1+2\alpha)/2}.
$$
Collecting now \eqref{sus12} and the inequality above, for every $|\lambda|\geq 1$ we have
\begin{align*}
\|z\|^2_\H&\leq 2\|v\|^2+c \|z\|_\H^{3/2}\|\hat z\|_\H^{1/2} + c\|z\|_\H \|\hat z \|_\H \\
&\leq c |\lambda|^\alpha \|z\|_\H^{3/2}\|\hat z\|_\H^{1/2} + c \|z\|_\H^{(3-2\alpha)/2}\|\hat z\|_\H^{(1+2\alpha)/2}
+ c\|z\|_\H \|\hat z \|_\H,
\end{align*}
Finally, we estimate the first two terms on the right-hand side making use of the Young inequality with
conjugate exponents $(\frac{4}{3},4)$ and $(\frac{4}{3-2\alpha},\frac{4}{1+2\alpha})$, respectively. This leads to
$$
\|z\|^2_\H\leq \frac{1}{4}\|z\|^2_\H + c|\lambda|^{4\alpha}\|\hat z\|_\H^2 +c\|z\|_\H \|\hat z \|_\H
\leq \frac{1}{2}\|z\|^2_\H + c|\lambda|^{4\alpha}\|\hat z\|_\H^2
$$
for all $|\lambda|\geq 1$, which readily implies \eqref{goalsus}.
\qed

\medskip
\noindent
{\sc Case ${\alpha >1}$.} Exploiting \eqref{furcont}, \eqref{sus1}
and the  H\"older inequality, we infer that
\begin{align*}
\|v\|^2 &\leq
\left(\int_{\sigma(A)} \frac{1}{s} \,\d \mu_{v,v} (s) \right)^{\frac{1}{2}}
\left(\int_{\sigma(A)} s \,\d \mu_{v,v} (s) \right)^{\frac{1}{2}}\\
&\leq \left(\int_{\sigma(A)} \frac{1}{s^\alpha} \,\d \mu_{v,v} (s) \right)^{\frac{1}{2\alpha}}\|v\|^{(\alpha-1)/\alpha}\|v\|_1\\\noalign{\vskip1.7mm}
&\leq c \|\sqrt{f(A)}v\|^{1/\alpha}\|z\|_\H^{(\alpha-1)/\alpha} \|v\|_1
\\\noalign{\vskip3mm}
&\leq c \|v\|_1 \|z\|_\H^{(2\alpha-1)/2\alpha}\|\hat z\|_\H^{1/2\alpha}.
\end{align*}
Recalling \eqref{sus11}, we get the control
$$
\|v\|^2 \leq c |\lambda|\|z\|_\H^{(4\alpha-1)/2\alpha}\|\hat z\|_\H^{1/2\alpha}
+ c\|z\|_\H^{(2\alpha-1)/2\alpha}\|\hat z\|_\H^{(2\alpha+1)/2\alpha}.
$$
At this point, from \eqref{sus12} and the inequality above, we obtain
\begin{align*}
\|z\|^2_\H&\leq 2\|v\|^2+c \|z\|_\H^{3/2}\|\hat z\|_\H^{1/2} + c\|z\|_\H \|\hat z \|_\H \\
&\leq c |\lambda|\|z\|_\H^{(4\alpha-1)/2\alpha}\|\hat z\|_\H^{1/2\alpha}
+ c\|z\|_\H^{(2\alpha-1)/2\alpha}\|\hat z\|_\H^{(2\alpha+1)/2\alpha} +
c \|z\|_\H^{3/2}\|\hat z\|_\H^{1/2}+ c\|z\|_\H \|\hat z \|_\H.
\end{align*}
Estimating the first three terms on the right-hand side making use of the Young inequality with
conjugate exponents
$(\frac{4\alpha}{4\alpha-1},4\alpha),(\frac{4\alpha}{2\alpha-1},\frac{4\alpha}{2\alpha+1})$
and $(\frac{4}{3},4)$, respectively, we end up with
$$
\|z\|^2_\H\leq \frac{1}{4}\|z\|^2_\H + c|\lambda|^{4\alpha}\|\hat z\|_\H^2+c\|z\|_\H \|\hat z \|_\H
\leq \frac{1}{2}\|z\|^2_\H + c|\lambda|^{4\alpha}\|\hat z\|_\H^2
$$
for all $|\lambda|\geq 1$. The latter implies at once \eqref{goalsus}.
\end{proof}

\begin{proof}[Conclusion of the proof of Theorem~\ref{decayrateteo}]
The claim merely follows by collecting the above
Lemma~\ref{aboveres} and
Theorem~\ref{BTpol}.
\end{proof}

\section{Proof of Theorem \ref{decayrateopt}}
\label{segundo}

\noindent
The key tool is the next result of Batty and Duyckaerts~\cite{Batty1}
(see also \cite[Theorem 4.4.14]{BattyLibroBis}).

\begin{theorem}[Batty-Duyckaerts]
\label{BDUY}
Assume that $S(t)$ is semiuniformly stable.
Then for any (strictly) decreasing continuous function $h:[0,\infty)\to \R^+$
vanishing at infinity
and satisfying $\|S(t)\A^{-1}\|_{L(\H)} \leq h(t)$, there exists $C>0$
such that
$$
\|(\ii\lambda-\A)^{-1}\|_{L(\H)} \leq C h^{-1}\Big(\frac1{2|\lambda|}\Big)
$$
for all $|\lambda|$ sufficiently large.
\end{theorem}

In order to apply this abstract theorem, we need to find
a (polynomial) control of the operator norm $\|(\ii\lambda-\A)^{-1}\|_{L(\H)} $. This is provided by the next lemma.

\begin{lemma}
\label{belowres}
Within the assumptions of Theorem \ref{decayrateopt},
we have ($\sigma(\A)\cap \ii\R=\emptyset$ and)
$$
\limsup_{\lambda\to\infty} \lambda^{-2\beta} \|(\ii \lambda - \A)^{-1}\|_{L(\H)}>0.
$$
\end{lemma}

\begin{proof}
Take $s_n \in \sigma(A)$ with $s_n\to\infty$ (this is possible being the operator $A$ unbounded).
Since $\sigma(\A)\cap \ii\R=\emptyset$ due to Theorem \ref{battscrit}, we introduce the sequence
$$
\eta_n = \|(\ii \sqrt{s_n} - \A)^{-1}\|_{L(\H)}.
$$
Exploiting the functional calculus of $A$, and arguing exactly as in the proof of \cite[Lemma~9.2]{DDP}, we can
find unit vectors $w_n\in H$ such that
\begin{align*}
&r_n^1 \doteq A w_n - s_n w_n \\
&r_n^2 \doteq f(A) w_n - f(s_n) w_n
\end{align*}
satisfy the bounds
\begin{equation}
\label{restecont}
\|r_n^\imath\| \leq \frac{1}{n (1+\eta_n)}, \quad \,\, \imath=1,2.
\end{equation}
Next, we define the further unit vectors
$$
\hat z_n = (0,w_n) \in\H
$$
and we consider the resolvent equation
$$
\ii \sqrt{s_n} z_n -\A z_n= \hat z_n
$$
which admits a unique solution
$$z_n = (u_n,v_n)=(\ii \sqrt{s_n} -\A)^{-1} \hat z_n.$$ In particular,
since $\|\hat z_n\|_\H=1$, we have
\begin{equation}
\label{unit}
\|z_n\|_\H = \|(\ii \sqrt{s_n} -\A)^{-1} \hat z_n\|_{\H} \leq \eta_n.
\end{equation}
Writing the resolvent equation componentwise, we obtain the system
\begin{align*}
&\ii\sqrt{s_n} u_n - v_n=0,\\
&\ii\sqrt{s_n} v_n + A u_n + f(A) v_n =  w_n.
\end{align*}
Substituting the first equation into the second one, we find
\begin{equation}
\label{eqz}
-s_n v_n + A v_n + \ii\sqrt{s_n} f(A) v_n =  \ii\sqrt{s_n} w_n.
\end{equation}
It is readily seen that $v_n$ can be written in the form
$$
v_n = \zeta_n w_n + q_n
$$
for some $\zeta_n\in\CC$ and some vectors $q_n \perp w_n$. Due to \eqref{unit}, the controls
\begin{equation}
\label{eqzbound}
|\zeta_n| \leq \|z_n\|_\H \leq \eta_n \and \|q_n\| \leq \|z_n\|_\H \leq \eta_n
\end{equation}
hold. At this point, we multiply \eqref{eqz} by $w_n$ in $H$. Exploiting the fact that $w_n$ and $q_n$ are orthogonal, and
after straightforward calculations, we get the equality
(recall that $f(s)\neq0$ for all $s \in \sigma(A)$ due to Theorem \ref{tsus})
$$
\zeta_n  = \frac{1}{f(s_n)} +\rho_n
$$
having set
$$
\rho_n=\frac{- \zeta_n \l r_n^1 , w_n \r - \l q_n , r_n^1 \r
-\ii\sqrt{s_n} \zeta_n \l r_n^2 , w_n\r -\ii\sqrt{s_n} \l q_n , r_n^2\r}{\ii\sqrt{s_n} f(s_n)}.
$$
Note that, in light of \eqref{restecont} and \eqref{eqzbound}, for $n$ large enough
$$
|\rho_n| \leq \frac{2\eta_n(1+\sqrt{s_n})}{n(1+\eta_n)f(s_n)\sqrt{s_n}}
\leq \frac{4}{nf(s_n)}.
$$
Accordingly, for $n$ large enough
$$
|\zeta_n|  \geq \frac{1}{f(s_n)} - |\rho_n|
\geq \frac{1}{2f(s_n)}.
$$
Finally, making use of \eqref{unit}, the inequality above and \eqref{furcont2}, we arrive at
$$
\eta_n  \geq \|z_n\|_\H \geq |\zeta_n| \geq
\frac{1}{2f(s_n)}\geq c s_n^\beta
$$
for some structural constant $c>0$. Recalling the definition of $\eta_n$,
from the latter estimates we conclude that
$$
\limsup_{n\to\infty} s_n^{-2\beta} \|(\ii s_n - \A)^{-1}\|_{L(\H)}>0,
$$
and the thesis follows.
\end{proof}
\begin{proof}[Conclusion of the proof of Theorem \ref{decayrateopt}]
Assume by contradiction that
$$
\psi(t) = {\rm o} (t^{-\frac{1}{2\beta}})\quad\, \text{as } t \to \infty,
$$
and let $h:[0,\infty)\to \R^+$ be a (strictly) decreasing continuous function satisfying
$$\psi(t) \leq h(t)\and h(t)={\rm o} (t^{-\frac{1}{2\beta}}).$$
One may take for instance
$$h(t)=\ds\sup_{s\geq t}\psi(s)+\e^{-t}.$$
Invoking now Theorem \ref{BDUY}, there exists $C>0$ such that
$$
\|(\ii\lambda-\A)^{-1}\|_{L(\H)} \leq C h^{-1}\Big(\frac1{2|\lambda|}\Big)
$$
for every $|\lambda|$ large enough. Since
$$
h^{-1}\Big(\frac1{2|\lambda|}\Big) = {\rm o} (|\lambda|^{2\beta})\quad\,\text{as } |\lambda| \to \infty,
$$
we conclude that
$$
\|(\ii\lambda-\A)^{-1}\|_{L(\H)}= {\rm o} (|\lambda|^{2\beta})\quad\,\text{as } |\lambda| \to \infty,
$$
contradicting Lemma \ref{belowres}.
\end{proof}

\section{Applications}
\label{SECAPPL}

\noindent
We now apply the results obtained so far to the examples presented in Section \ref{seccont}, to which
we address the reader for the notation. In what follows, we denote by
$$0<\lambda_1<\lambda_2\leq \ldots \leq \lambda_n \to \infty$$
the sequence of eigenvalues of the Laplace-Dirichlet operator $L$.

\subsection{Abstract wave equations with fractional damping}
In view of Theorem \ref{EU}, equation~\eqref{polyabseq} generates a contraction semigroup
$$
S_\vartheta(t)=\e^{t\A{_\vartheta}} : \H \to \H.
$$
Here, $\A_\vartheta$ is the particular instance of $\A$ corresponding to the choice $f(s)=s^\vartheta$, namely
$$
\A_\vartheta
\left(\begin{matrix}
u\\
v
\end{matrix}\right)
= \left(
\begin{matrix}
v\\
- A u - A^\vartheta v
\end{matrix}
\right)
$$
with domain
$$
\D(\A_\vartheta)=
\left\{ (u,v)\in\H \left|
\begin{array}{c}
v\in \HH^1\\
Au + A^\vartheta v \in \HH
\end{array}\right.
\right\}.
$$
Note that
$\sup_{s\in\sigma(A)}f(s)/\sqrt{s}<\infty$ if and only if $\vartheta\leq \tfrac12$ or $A$ is a bounded operator. In this
situation (and only in this situation), the domain factorizes as
$$
\D(\A_\vartheta) = H^2 \times H^1.
$$
Exploiting Theorems \ref{zerospect} and \ref{spec}, we obtain a precise description of the spectrum of $\A_\vartheta$.

\begin{theorem}
\label{specabspolywa}
The following hold:
\begin{itemize}
\item[\rm (i)] The operator $\A_\vartheta$ is bijective if and only if $\vartheta\leq1$ or $A$ is a bounded operator.
\smallskip
\item[\rm (ii)] We have
$$
\sigma(\A_\vartheta)\setminus\{ 0\} =
\begin{cases}
\bigcup_{s\in\sigma(A)} \big\{\xi_s^{\pm} \big\}
\cup \{ -1\} & {\text{ if } } \vartheta= 1 \text{ and $A$ unbounded},\\
\noalign{\vskip2mm}
\bigcup_{s\in\sigma(A)} \big\{\xi_s^{\pm} \big\}
& {\text{ otherwise, } }
\end{cases}
$$
where
$$
\xi_s^{\pm}=
\begin{cases}
\ds-\frac{s^\vartheta}{2} \pm \frac{\sqrt{s^{2\vartheta}-4s}}{2} &\quad\text{if }\, s^{\vartheta-\tfrac12}\geq 2,\\
\noalign{\vskip1mm}
\ds -\frac{s^\vartheta}{2} \pm \ii\frac{\sqrt{4s-s^{2\vartheta}}}{2} &\quad\text{if }\, s^{\vartheta-\tfrac12}< 2.
\end{cases}
$$
\end{itemize}
\end{theorem}

Indeed, with reference to \eqref{lambdaset}, the set $\Lambda$ is nonempty if and only if $\vartheta=1$ and the operator
$A$ is unbounded. Besides, if the latter conditions hold, then $\Lambda = \{ 1\}$.

Theorem \ref{specabspolywa}, together with Remark \ref{remeig},
produce an immediate corollary, which provides a characterization of the spectrum of the wave equation~\eqref{concvardamp},
where $A$ is the Laplace-Dirichlet operator.

\begin{corollary}
\label{corcorcor}
Let $A=L$.
Then, the spectrum of the corresponding operator $\A_\vartheta$ is countable and is given by
$$
\sigma(\A_\vartheta)=
\begin{cases}
\bigcup_{n=1}^\infty \big\{\xi_{\lambda_n}^{\pm} \big\}
\cup \{ 0\} & {\text{ if } } \vartheta> 1,\\
\noalign{\vskip2mm}
\bigcup_{n=1}^\infty \big\{\xi_{\lambda_n}^{\pm} \big\}
\cup \{ -1\} & {\text{ if } } \vartheta= 1,\\
\noalign{\vskip2mm}
\bigcup_{n=1}^\infty \big\{\xi_{\lambda_n}^{\pm} \big\}
&  {\text{ if } } \vartheta< 1.
\end{cases}
$$
Besides, the numbers $\xi_{\lambda_n}^{\pm}$
are all eigenvalues of $\A_\vartheta$.
\end{corollary}

Coming back to more general equation~\eqref{polyabseq},
the decay properties of the related semigroup $S_\vartheta(t)$ can
be immediately inferred from the results of Sections~\ref{stabilitySection}-\ref{POLYstabilitySection},
observing that $\ZZ=\emptyset$, due to the choice of the function $f$.
It is readily seen that, if the operator $A$ is bounded, then $S_\vartheta(t)$ is
exponentially stable for every $\vartheta\in\R$ (since condition \eqref{condexpstab} is always satisfied). The more interesting case
when $A$ is unbounded is summarized in the next theorem.

\begin{theorem}
Let $A$ be unbounded. Then, the following hold:
\begin{itemize}
\item[\rm(i)] $S_\vartheta(t)$ is stable for every $\vartheta\in\R$.
\smallskip
\item[\rm(ii)] $S_\vartheta(t)$ is exponentially stable if and only if
 $\vartheta \in [0,1]$.
\smallskip
\item[\rm(iii)] $S_\vartheta(t)$ is semiuniformly stable if and only if $\vartheta \leq 1$.
\item[\rm(iv)] If $\vartheta<0$, then $\psi(t)=\|S_\vartheta(t)\A_\vartheta^{-1}\|_{L(\H)}$ decays polynomially
as $t^{-\frac{1}{2|\vartheta|}}$,
and such a decay rate is optimal.
\end{itemize}
\end{theorem}

\subsection{Beams and plates without rotational inertia}
Here we consider \eqref{nonrotrwe}, that is, equation~\eqref{platefrac} with $\omega=0$.
Although in this situation the picture is formally equivalent to the previous example
with $\tfrac{\vartheta}{2}$ in place of $\vartheta$, for completeness we provide a detailed description of the results.
Invoking Theorem \ref{EU}, equation~\eqref{nonrotrwe} generates a contraction semigroup
$$
S_\vartheta^0(t)=\e^{t\A{_\vartheta^0}} : \H^0 \to \H^0
$$
on the space
$$\H^0= V^2 \times V.$$
Here, $\A_\vartheta^0$ is the particular instance of the operator
$\A$ obtained by choosing $H=V$, $A=L^2$ and $f(s)=s^\frac{\vartheta}{2}$, that is
$$
\A_\vartheta^0
\left(\begin{matrix}
u\\
v
\end{matrix}\right)
= \left(
\begin{matrix}
v\\
- L^2 u - L^\vartheta v
\end{matrix}
\right)
$$
with domain
$$
\D(\A_\vartheta^0)=
\left\{ (u,v)\in\H^0 \left|
\begin{array}{c}
v\in V^2\\
L^2 u + L^\vartheta v \in V
\end{array}\right.
\right\}.
$$
Being the operator $A=L^2$ unbounded, we have
$\sup_{s\in\sigma(A)}f(s)/\sqrt{s}<\infty$ if and only if $\vartheta\leq 1$. In
this situation (and only in this situation), the domain takes the form
$$
\D(\A_\vartheta^0) = V^4 \times V^2.
$$
Moreover, since the spectrum of the operator $L^2$ is entirely made by eigenvalues and reads
$$
\sigma(L^2) = \bigcup_{n=1}^\infty \{\lambda_n^2\},
$$
an exploitation of Theorems~\ref{zerospect} and \ref{spec}, together with Remark \ref{remeig},
yields a complete description of $\sigma(\A_\vartheta^0)$.

\begin{theorem}
\label{spettronorot}
The spectrum of $\A_\vartheta^0$ is countable and is given by
$$
\sigma(\A_\vartheta^0)=
\begin{cases}
\bigcup_{n=1}^\infty \big\{\xi_{\lambda_n^2}^\pm \big\}
\cup \{ 0\} & {\text{ if } } \vartheta> 2,\\
\noalign{\vskip2mm}
\bigcup_{n=1}^\infty \big\{\xi_{\lambda_n^2}^\pm \big\}
\cup \{ -1\} & {\text{ if } } \vartheta= 2,\\
\noalign{\vskip2mm}
\bigcup_{n=1}^\infty \big\{\xi_{\lambda_n^2}^\pm \big\}
&  {\text{ if } } \vartheta< 2.
\end{cases}
$$
Explicitly,
$$
\xi_{\lambda_n^2}^\pm=
\begin{cases}
\ds-\frac{\lambda_n^{\vartheta}}{2} \pm \frac{\sqrt{\lambda_n^{2\vartheta}-4\lambda_n^2}}{2} &\quad\text{if }\, \lambda_n^{\vartheta-1}\geq 2,\\
\noalign{\vskip1mm}
\ds -\frac{\lambda_n^{\vartheta}}{2} \pm \ii\frac{\sqrt{4\lambda_n^2-\lambda_n^{2\vartheta}}}{2} &\quad\text{if }\, \lambda_n^{\vartheta-1}< 2.
\end{cases}
$$
Besides, the numbers $\xi_{\lambda_n^2}^\pm$
are all eigenvalues of $\A_\vartheta^0$.
\end{theorem}

Here, we made use of the fact that the set $\Lambda$
defined in \eqref{lambdaset} is nonempty if and only if $\vartheta=2$, and in this case $\Lambda = \{ 1\}$.

Finally, the stability properties of the semigroup $S_\vartheta^0(t)$ can
be inferred from the results of Sections~\ref{stabilitySection}-\ref{POLYstabilitySection},
noting that $\ZZ=\emptyset$.

\begin{theorem}
The following hold:
\begin{itemize}
\item[\rm(i)] $S_\vartheta^0(t)$ is stable for every $\vartheta\in\R$.
\smallskip
\item[\rm(ii)] $S_\vartheta^0(t)$ is exponentially stable if and only if
 $\vartheta \in [0,2]$.
\smallskip
\item[\rm(iii)] $S_\vartheta^0(t)$ is semiuniformly stable if and only if $\vartheta \leq 2$.
\item[\rm(iv)] If $\vartheta<0$, then $\psi(t)=\|S_\vartheta^0(t)(\A_\vartheta^0)^{-1}\|_{L(\H^0)}$ decays polynomially
as $t^{-\frac{1}{|\vartheta|}}$,
and such a decay rate is optimal.
\end{itemize}
\end{theorem}

\subsection{Beams and plates with rotational inertia}
The last example concerns equation~\eqref{platefrac} with $\omega>0$, which can be rewritten in the form \eqref{rotrwe}.
Analogously to the previous examples, appealing to
Theorem \ref{EU} we infer that equation~\eqref{rotrwe} generates a contraction semigroup
$$
S_\vartheta^\omega(t)=\e^{t\A{_\vartheta^\omega}} : \H^\omega \to \H^\omega
$$
on the space
$$\H^\omega= V^2 \times V^1.$$
This time, $\A_\vartheta^\omega$ is the particular instance of
$\A$ corresponding to $H=V^1$,
$A = (1+\omega L)^{-1} L^2$, with $f$ given by \eqref{effe}.
Explicitly,
$$
\A_\vartheta^\omega
\left(\begin{matrix}
u\\
v
\end{matrix}\right)
= \left(
\begin{matrix}
v\\
- (1+\omega L)^{-1} L^2 u - (1+\omega L)^{-1} L^\vartheta v
\end{matrix}
\right)
$$
with domain
$$
\D(\A_\vartheta^\omega)=
\left\{ (u,v)\in\H^\omega \left|
\begin{array}{c}
v\in V^2\\
(1+\omega L)^{-1} L^2 u + (1+\omega L)^{-1} L^\vartheta v \in V^1
\end{array}\right.
\right\}.
$$
Since $A$ is an unbounded operator, a closer look to \eqref{effe} tells that
$\sup_{s\in\sigma(A)}f(s)/\sqrt{s}<\infty$ if and only if $\vartheta \leq \tfrac32$. In
this situation (and only in this situation), the domain factorizes as
$$
\D(\A_\vartheta^\omega) = V^3 \times V^2.
$$
Similarly to the case $\omega=0$, the spectrum of $A$ is entirely made by eigenvalues and is given by
$$
\sigma(A) = \bigcup_{n=1}^\infty \bigg\{\frac{\lambda_n^2}{1+\omega \lambda_n}\bigg\}.
$$
Accordingly, making use of \eqref{effe}, together with Theorems~\ref{zerospect} and \ref{spec} and Remark \ref{remeig},
we readily get a complete description of the spectrum of $\A_\vartheta^\omega$.

\begin{theorem}
\label{spettrorot}
The spectrum of $\A_\vartheta^\omega$ is countable and is given by
$$
\sigma(\A_\vartheta^\omega)=
\begin{cases}
\bigcup_{n=1}^\infty \big\{\xi_{\nu_n}^{\pm} \big\}
\cup \{ 0\} & {\text{ if } } \vartheta> 2,\\
\noalign{\vskip2mm}
\bigcup_{n=1}^\infty \big\{\xi_{\nu_n}^{\pm} \big\}
\cup \{ -1\} & {\text{ if } } \vartheta= 2,\\
\noalign{\vskip2mm}
\bigcup_{n=1}^\infty \big\{\xi_{\nu_n}^{\pm} \big\}
&  {\text{ if } } \vartheta< 2,
\end{cases}
$$
having set
$$\nu_n = \frac{\lambda^2_n }{1+\omega \lambda_n}.$$
Explicitly,
$$
\xi_{\nu_n}^{\pm}=
\begin{cases}
\ds-\frac{f(\nu_n)}{2} \pm \frac{\sqrt{[f(\nu_n)]^2-4\nu_n}}{2} &\quad\text{if } f(\nu_n)\geq 2\sqrt{\nu_n}\,,\\
\noalign{\vskip1mm}
\ds -\frac{f(\nu_n)}{2} \pm \ii\frac{\sqrt{4\nu_n-[f(\nu_n)]^2}}{2} &\quad\text{if } f(\nu_n)< 2\sqrt{\nu_n}\,,
\end{cases}
$$
where $f$ is given by \eqref{effe}.
Besides, the numbers $\xi_{\nu_n}^{\pm}$
are all eigenvalues of $\A_\vartheta^\omega$.
\end{theorem}

Exactly as in the previous example, we have exploited the fact that the set $\Lambda$
defined in \eqref{lambdaset} is nonempty if and only if $\vartheta=2$, and in this case $\Lambda = \{ 1\}$.

We conclude by summarizing the stability properties of $S_\vartheta^\omega(t)$
which, again, can be readily inferred from the results
of Sections~\ref{stabilitySection}-\ref{POLYstabilitySection},
observing that $\ZZ=\emptyset$ and
$$
f(s) \sim \omega^{\vartheta-2} s^{\vartheta-1}\quad\,\, \text{as } s \to\infty.
$$

\begin{theorem}
The following hold:
\begin{itemize}
\item[\rm(i)] $S_\vartheta^\omega(t)$ is stable for every $\vartheta\in\R$.
\smallskip
\item[\rm(ii)] $S_\vartheta^\omega(t)$ is exponentially stable if and only if
 $\vartheta \in [1,2]$.
\smallskip
\item[\rm(iii)] $S_\vartheta^\omega(t)$ is semiuniformly stable if and only if $\vartheta \leq 2$.
\item[\rm(iv)] If $\vartheta<1$, then $\psi(t)=\|S_\vartheta^\omega(t)(\A_\vartheta^\omega)^{-1}\|_{L(\H^\omega)}$ decays polynomially
as $t^{-\frac{1}{2(1-\vartheta)}}$,
and such a decay rate is optimal.
\end{itemize}
\end{theorem}

\begin{remark}
Note that the presence of the rotational inertia $\omega>0$ changes
the exponential stability interval of $\vartheta$ from $[0,2]$ to $[1,2]$.
\end{remark}

\section{Further Developments}

\noindent
We finally discuss some possible developments,
which might be deepened in future works.

\medskip
\noindent
{\bf I.} As pointed out in Remark \ref{remopen}, it would be interesting to investigate
the stability of $S(t)$ where the operator $A^{-1}$ is not compact and
$\ZZ$ is an uncountable set with null spectral measure.
As shown in the proof of Theorem \ref{stability}, if $E_A(\ZZ)=0$ then it is always
true that no eigenvalues of $\A$ lie on the imaginary axis. Nevertheless,
Corollary~\ref{speccoro} tells that the set $\sigma(\A)\cap \ii\R$ is
uncountable. Hence, in this situation, the Arendt-Batty-Lyubich-V\~u stability criterion cannot be applied.

\medskip
\noindent
{\bf II.} Another problem concerns the behavior at infinity of the resolvent operator
$(\ii\lambda - \A)^{-1}$ on the imaginary axis when the exponential stability of $S(t)$ occurs.
Due to the Gearhart-Pr\"{u}ss theorem \cite{GER,Pru}, such a resolvent operator
is (defined and) bounded on the whole imaginary axis.
If for instance
$$\|(\ii\lambda - \A)^{-1}\|_{L(\H)}={\rm O}\bigg(\frac{1}{\log|\lambda|}\bigg)\quad\,\, \text{as } |\lambda|\to\infty$$
then $S(t)$
is eventually differentiable (see e.g.\ \cite[Theorem 4.9]{Pazy}), while
if
$$\|(\ii\lambda - \A)^{-1}\|_{L(\H)}={\rm O}\bigg(\frac{1}{|\lambda|}\bigg)\quad\,\, \text{as } |\lambda|\to\infty$$ then $S(t)$
is analytic (see e.g.\ \cite[Theorem 1.3.3]{LZH}). Among other reasons,
such regularity properties have a certain relevance since eventually differentiable semigroups or analytic semigroups are
know to fulfill the {\it spectrum determined growth} (SDG) condition (see e.g.\ \cite[Corollary 3.12]{ENGNAG}).
In the notation of Section \ref{ExpSec}, this means that the growth bound $\omega_*$ of $S(t)$ equals the spectral bound $\sigma_*$
of its infinitesimal generator $\A$. Note that, in the proof of the necessity part of Theorem \ref{stabexpteo},
we have already shown that the SDG condition is satisfied
with $\omega_*=\sigma_*=0$ whenever condition \eqref{condexpstab} fails.

\medskip
\noindent
{\bf III.} An intriguing and possibly challenging task would be to investigate
semiuniform (or semiuniform-like) decay rates of $S(t)$ which are not necessarily of
polynomial type, making use of recent abstract results
obtained in \cite{BattyJEMS} (see also \cite{SEI}) dealing with
fine decay scales of strongly continuous semigroups.

\bigskip
\section*{Appendix: Portraits of the Spectra}

\theoremstyle{plain}
\renewcommand{\thesubsection}{A.\arabic{subsection}}
\setcounter{equation}{0}
\setcounter{subsection}{0}
\renewcommand{\theequation}{A.\arabic{equation}}

\noindent
We illustrate some particular instances of the spectra of the operators $\A_\vartheta$, $\A_\vartheta^0$
and $\A_\vartheta^\omega$ discussed in Section \ref{SECAPPL}.

\smallskip
\noindent
{\bf Portraits of $\boldsymbol{\sigma(\A_\vartheta)}$.}
Choosing $H=V=L^2(0,\pi)$ and $A=L$, the eigenvalues $\lambda_n$
are equal to
$$\lambda_n= n^2,\quad\,\, n=1,2,3,\ldots
$$
Accordingly, the eigenvalues $\xi_{\lambda_n}^{\pm}$ of $\A_\vartheta$ take the form
$$
\xi_{\lambda_n}^{\pm}=
\begin{cases}
\ds-\frac{n^{2\vartheta}}{2} \pm \frac{\sqrt{n^{4\vartheta}-4n^2}}{2} &\quad\text{if }\, n^{2\vartheta-1}\geq 2,\\
\noalign{\vskip1mm}
\ds -\frac{n^{2\vartheta}}{2} \pm \ii\frac{\sqrt{4n^2-n^{4\vartheta}}}{2} &\quad\text{if }\, n^{2\vartheta-1}< 2.
\end{cases}
$$
Making use of Corollary \ref{corcorcor} and the software {\tt Mathematica}{\tiny \textregistered},
we have the following pictures of $\sigma(\A_\vartheta)$,  corresponding to the cases $\vartheta=-1,0,1$.

\begin{figure}[ht]\begin{center}
\includegraphics[width=\wws]{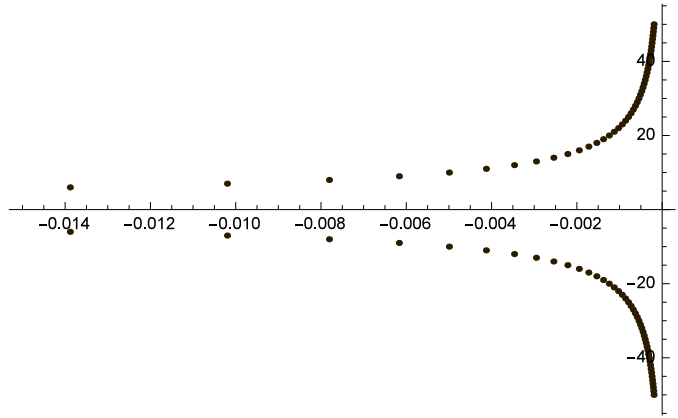}
\caption{\small{The case $\vartheta=-1$ [``subdamped" wave equation].}}
\end{center}
\end{figure}
\begin{figure}[ht]\begin{center}

\medskip

\includegraphics[width=\wws]{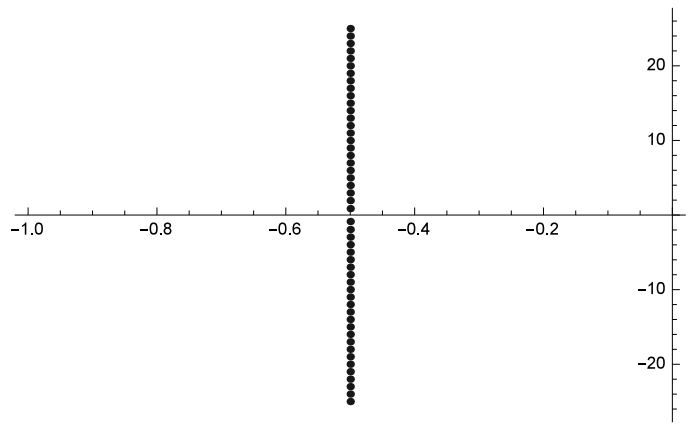}
\caption{\small{The case $\vartheta=0$ [weakly damped wave equation].}}
\end{center}
\end{figure}
\begin{figure}[ht]\begin{center}

\medskip

\includegraphics[width=\wwp]{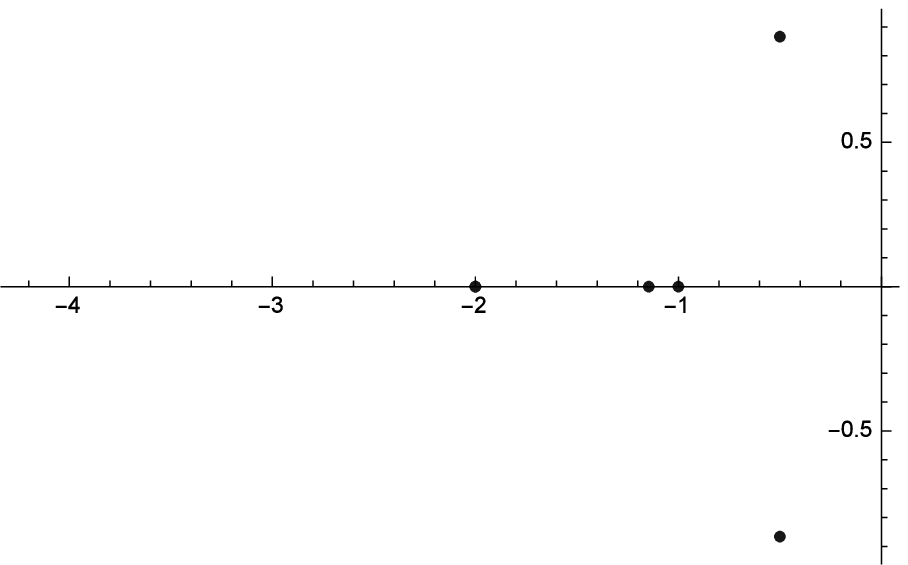}\qquad\,\, \includegraphics[width=\wwp]{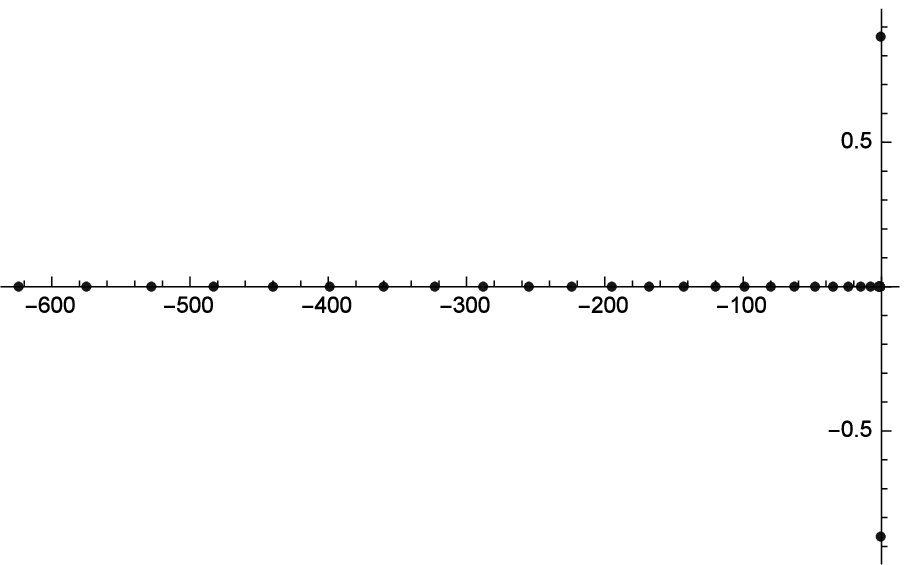}
\caption{\small{The case $\vartheta=1$ [strongly damped wave equation]. Behavior around zero (left)
and global behavior (right).}}
\end{center}
\end{figure}

\newpage
\noindent
{\bf Portraits of $\boldsymbol{\sigma(\A_\vartheta^0)}$.}
Choosing $H=V=L^2(0,\pi)$, the eigenvalues $\lambda_n^2$
of the operator $A=L^2$
are equal to
$$\lambda_n^2= n^4,\quad\,\, n=1,2,3,\ldots
$$
Therefore, the eigenvalues $\xi_{\lambda_n^2}^{\pm}$ of $\A_\vartheta^0$ are given by
$$
\xi_{\lambda_n^2}^{\pm}=
\begin{cases}
\ds-\frac{n^{2\vartheta}}{2} \pm \frac{\sqrt{n^{4\vartheta}-4n^4}}{2} &\quad\text{if }\, n^{2(\vartheta-1)}\geq 2,\\
\noalign{\vskip1.3mm}
\ds -\frac{n^{2\vartheta}}{2} \pm \ii\frac{\sqrt{4n^4-n^{4\vartheta}}}{2} &\quad\text{if }\, n^{2(\vartheta-1)}< 2.
\end{cases}
$$
Making use
of Theorem \ref{spettronorot} and the software {\tt Mathematica}{\tiny \textregistered},
we get the following pictures of $\sigma(\A_\vartheta^0)$, corresponding to the choices $\vartheta=-1,0,1$.

\begin{figure}[ht]\begin{center}
\includegraphics[width=\wws]{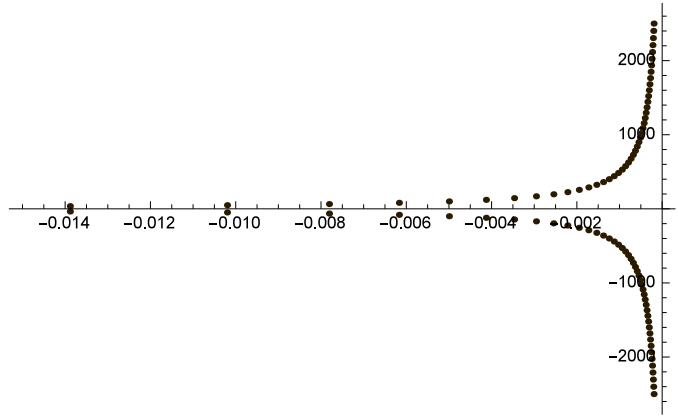}
\caption{\small{The case $\vartheta=-1$ [``subdamped" beam equation].}}
\end{center}
\end{figure}

\medskip

\begin{figure}[ht]\begin{center}
\includegraphics[width=\wws]{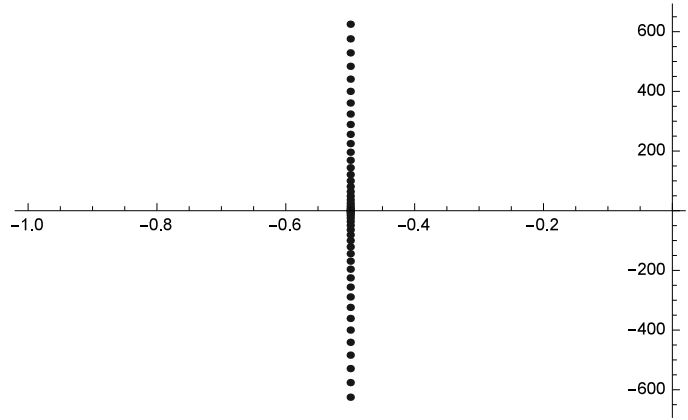}
\caption{\small{The case $\vartheta=0$ [beam equation with frictional damping].}}
\end{center}
\end{figure}

\medskip

\begin{figure}[ht]\begin{center}
\includegraphics[width=\wwp]{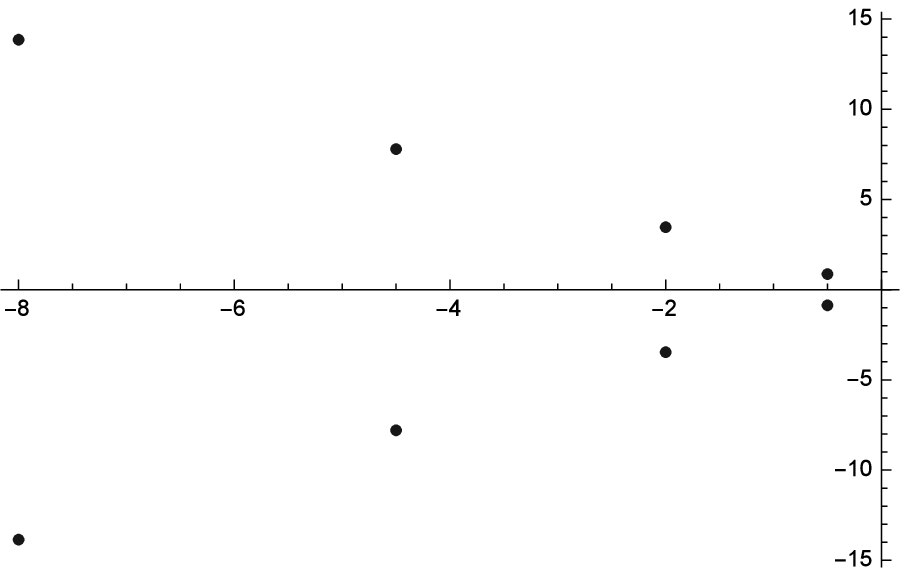}\qquad\,\,\includegraphics[width=\wwp]{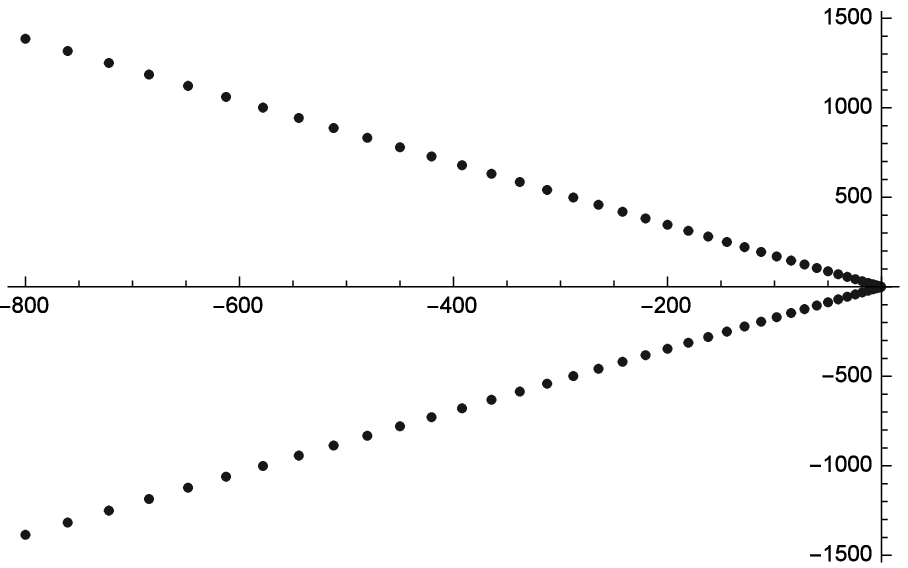}
\caption{\small{The case $\vartheta=1$ [beam equation with Kelvin-Voigt damping]. Behavior around zero (left)
and global behavior (right).}}
\end{center}
\end{figure}

\newpage

\noindent
{\bf Portraits of $\boldsymbol{\sigma(\A_\vartheta^\omega)}$.}
Choosing $H=V^1=H_0^1(0,\pi)$,
 the eigenvalues $\nu_n$
of the operator $A = (1+ \omega L)^{-1} L^2$ are equal to
$$
\nu_n = \frac{n^4}{1+\omega n^2}.
$$
Hence, the eigenvalues $\xi_{\nu_n}^\pm$ of $\A_\vartheta^\omega$ read
$$
\xi_{\nu_n}^{\pm}=
\begin{cases}
\ds-\frac{1}{2}\Big[f\big(\tfrac{n^4}{1+\omega n^2}\big) \mp
\sqrt{\big[f\big(\tfrac{n^4}{1+\omega n^2}\big)\big]^2-\tfrac{4n^4}{1+\omega n^2}}\Big]
&\quad\text{if } f\big(\tfrac{n^4}{1+\omega n^2}\big)\geq \frac{2n^2}{\sqrt{1+\omega n^2}},\\
\noalign{\vskip2.3mm}
\ds -\frac{1}{2}\Big[f\big(\tfrac{n^4}{1+\omega n^2}\big) \mp \ii
\sqrt{\tfrac{4n^4}{1+ \omega n^2}-\big[f\big(\tfrac{n^4}{1+\omega n^2}\big)\big]^2}\Big]
&\quad\text{if } f\big(\tfrac{n^4}{1+\omega n^2}\big)< \frac{2n^2}{\sqrt{1+\omega n^2}},
\end{cases}
$$
where $f$ is given by \eqref{effe}.
Making use of Theorem \ref{spettrorot} and
the software {\tt Mathematica}{\tiny \textregistered},
we obtain the following pictures of $\sigma(\A_\vartheta^\omega)$, corresponding to the choices $\vartheta=0,1$.

\begin{figure}[ht]\begin{center}
\includegraphics[width=\wws]{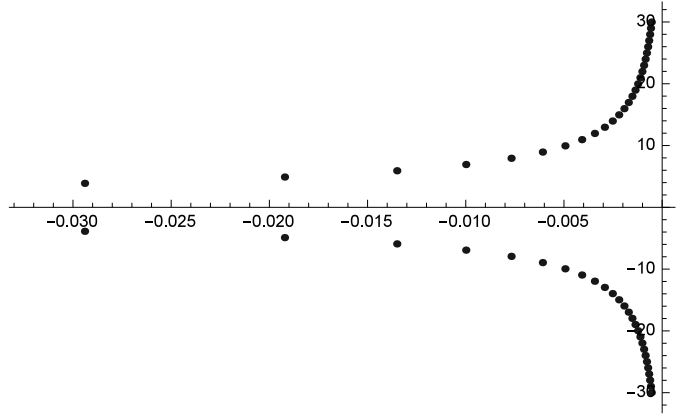}
\caption{\small{The case $\vartheta=0$ and $\omega=1$ [beam equation with rotational inertia and frictional damping].}}
\end{center}
\end{figure}

\medskip

\begin{figure}[ht]\begin{center}
\includegraphics[width=\wwp]{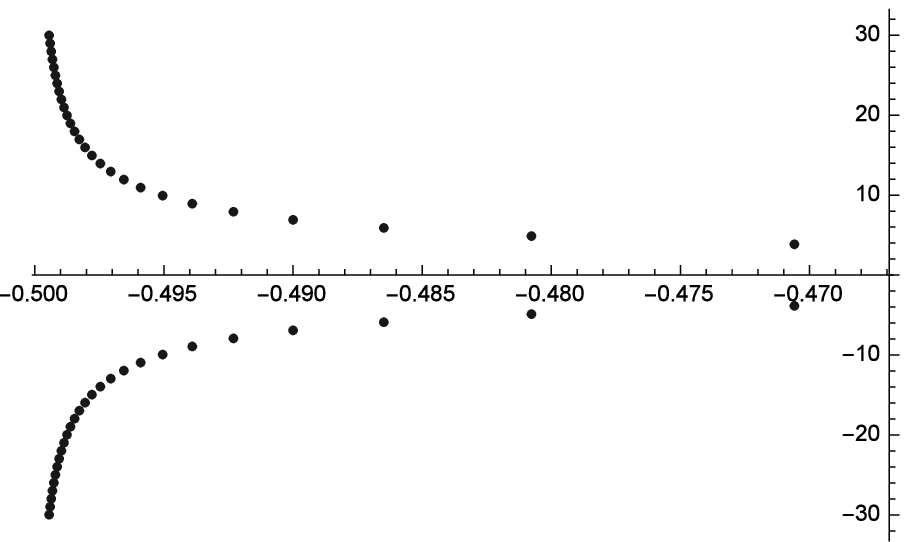}\qquad\,\,\includegraphics[width=\wwp]{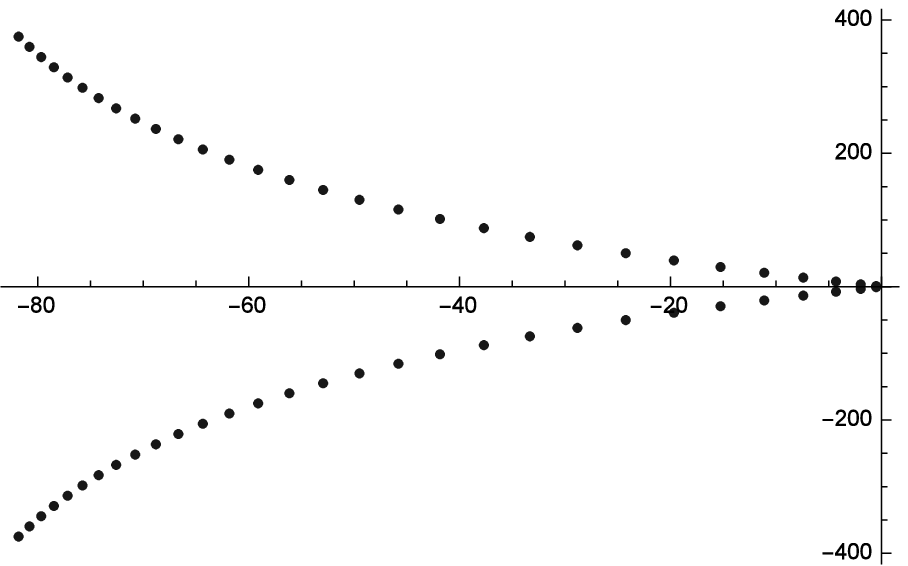}
\caption{\small{The case $\vartheta=1$ with $\omega=1$ (left)
and $\omega=1/200$ (right) [beam equation with rotational inertia and Kelvin-Voigt damping].
Note that the spectrum becomes close to two straight lines as $\omega\to0$. Compare with Figure~6.}}
\end{center}
\end{figure}



\end{document}